\documentclass[11pt]{amsart}
\usepackage{tikz}
\usepackage{tikz-cd}
\allowdisplaybreaks

\usepackage{amsmath,amsfonts,amsthm,amssymb,graphicx,tikz,color}
\usepackage[all,2cell,ps]{xy}

\theoremstyle{plain}
\newtheorem{thm}{Theorem}[section]
\newtheorem{lemma}[thm]{Lemma}
\newtheorem{prop}[thm]{Proposition}
\newtheorem{cor}[thm]{Corollary}

\newtheorem{qtn}[thm]{Question}

\theoremstyle{definition}
\newtheorem{defn}[thm]{Definition}

\newtheorem{ex}[thm]{Example}

\newtheorem{rem}[thm]{Remark}

\theoremstyle{remark}

 \DeclareMathOperator{\Out}{Out}
 
\DeclareMathOperator{\SL}{SL} \DeclareMathOperator{\PSL}{PSL} 
\DeclareMathOperator{\GL}{GL} 
 \DeclareMathOperator{\PO}{PO}
 \DeclareMathOperator{\SO}{SO} 
 
\DeclareMathOperator{\PU}{PU} 
\DeclareMathOperator{\SU}{SU}

 \DeclareMathOperator{\Ad}{Ad}

\DeclareMathOperator{\inn}{inn}

\DeclareMathOperator{\Isom}{Isom}

\newcommand{\al}{\alpha}

\newcommand{\sig}{\sigma}

\newcommand{\wh}{\widehat}

\newcommand{\bs}{\backslash}

\newcommand{\bbB}{\mathbb{B}}
\newcommand{\bbC}{\mathbb{C}}

\newcommand{\bbH}{\mathbb{H}}

\newcommand{\bbP}{\mathbb{P}}
\newcommand{\bbQ}{\mathbb{Q}}
\newcommand{\bbR}{\mathbb{R}}

\newcommand{\bbZ}{\mathbb{Z}}

\newcommand{\bfA}{\mathbf{A}}
\newcommand{\bfB}{\mathbf{B}}

\newcommand{\bfG}{\mathbf{G}}
\newcommand{\bfH}{\mathbf{H}}
\newcommand{\bfI}{\mathbf{I}}
\newcommand{\bfJ}{\mathbf{J}}

\newcommand{\bfL}{\mathbf{L}}

\newcommand{\bfS}{\mathbf{S}}

\newcommand{\bfV}{\mathbf{V}}

\newcommand{\frakg}{\mathfrak{g}}
\newcommand{\frakh}{\mathfrak{h}}

\newcommand{\frakj}{\mathfrak{j}}

\newcommand{\frakl}{\mathfrak{l}}

\newcommand{\fraks}{\mathfrak{s}}

\newcommand{\fraku}{\mathfrak{u}}

\newcommand{\gam}{\gamma}
\newcommand{\Gam}{\Gamma}

\newcommand{\Del}{\Delta}

\newcommand{\lam}{\lambda}
\newcommand{\Lam}{\Lambda}

\DeclareMathOperator{\U}{U}

\DeclareMathOperator{\M}{M}

\DeclareMathOperator{\Tr}{Tr}

\DeclareMathOperator{\supp}{supp}

\newcommand{\conj}{\overline}
\newcommand{\wt}{\widetilde}

\newcommand{\N}{\ensuremath{\mathbb{N}}}
\newcommand{\Q}{\ensuremath{\mathbb{Q}}}
\newcommand{\R}{\ensuremath{\mathbb{R}}}
\newcommand{\Z}{\ensuremath{\mathbb{Z}}}
\newcommand{\C}{\ensuremath{\mathbb{C}}}

\newcommand{\calC}{\ensuremath{\mathcal{C}}}

\newcommand{\calL}{\ensuremath{\mathcal{L}}}

\newcommand{\calO}{\ensuremath{\mathcal{O}}}

\newcommand{\calS}{\ensuremath{\mathcal{S}}}

\newcommand{\bH}{\ensuremath{\mathbf{H}}}

\newcommand{\ssm}{\smallsetminus}

\usetikzlibrary{arrows}

\title[Arithmeticity and superrigidity for $\SU(n,1)$]{Arithmeticity, Superrigidity and Totally Geodesic submanifolds of Complex Hyperbolic manifolds}
\author[U Bader]{Uri Bader}
\address{Weizmann Institute of Science}
\email{bader@weizmann.ac.il}
\author[D Fisher]{David Fisher}
\address{Department of Mathematics\\Rice University\\Houston, TX 77005}
\email{davidfisher@rice.edu}
\author[N Miller]{Nicholas Miller}
\address{Department of Mathematics\\University of Oklahoma\\Norman, OK 73019}
\email{nickmbmiller@ou.edu}
\author[M Stover]{Matthew Stover}
\address{Department of Mathematics\\Temple University\\Philadelphia, PA 19122}
\email{mstover@temple.edu}
\dedicatory{To Domingo Toledo with admiration on the occasion of his 75th birthday}
\begin{document}

\begin{abstract}
For $n \ge 2$, we prove that a finite volume complex hyperbolic $n$-manifold containing infinitely many maximal properly immersed totally geodesic submanifolds of real dimension at least two is arithmetic, paralleling our previous work for real hyperbolic manifolds. As in the real hyperbolic case, our primary result is a superrigidity theorem for certain representations of complex hyperbolic lattices. The proof requires developing new general tools not needed in the real hyperbolic case. Our main results also have a number of other applications.  For example, we prove nonexistence of certain maps between complex hyperbolic manifolds, which is related to a question of Siu, that certain hyperbolic $3$-manifolds cannot be totally geodesic submanifolds of complex hyperbolic manifolds, and that arithmeticity of complex hyperbolic manifolds is detected purely by the topology of the underlying complex variety, which is related to a question of Margulis. Our results also provide some evidence for a conjecture of Klingler that is a broad generalization of the Zilber--Pink conjecture.
\end{abstract}

\maketitle

\section{Introduction}\label{sec:Intro}

Throughout this paper, a \emph{geodesic submanifold} will always mean a properly immersed totally geodesic subspace, and a geodesic submanifold is called \emph{maximal} if it is not contained in a proper geodesic submanifold of larger dimension. Here and throughout a finite volume complex hyperbolic $n$-manifold will always have complex dimension $n$. In this paper, we will prove the following.

\begin{thm}[Arithmeticity] \label{thm:main}
Suppose that $n \ge 2$ and $M$ is a finite volume complex hyperbolic $n$-manifold containing infinitely many maximal totally geodesic submanifolds of real dimension at least $2$. Then $M$ is arithmetic.
\end{thm}

In previous work, we proved arithmeticity of \emph{real} hyperbolic manifolds containing infinitely many maximal geodesic subspaces of dimension at least two \cite[Thm.\ 1.1]{BFMS}, which answered a question independently due to Alan Reid and Curtis McMullen \cite{Curt}.  As in that case, Theorem~\ref{thm:main} can be restated purely in terms of homogeneous dynamics, where one obtains the exact same statement as \cite[Thm.\ 1.5]{BFMS}. We also note that, after learning of our result, Baldi and Ullmo recently gave a very different proof of Theorem \ref{thm:main} in the special case of totally geodesic complex subvarieties \cite{BaldiUllmo}.

To answer Reid and McMullen's question we introduced the notion of \emph{compatibility} of an algebraic group $\bfH$ over a local field $k$ with a semisimple Lie group $G$ and used this with $G = \SO_0(n,1)$ to prove a \emph{superrigidity theorem} for certain representations of real hyperbolic lattices \cite[Thm.\ 1.6]{BFMS}. However, for $G=\SU(n,1)$ superrigidity with compatible targets is not enough to prove arithmeticity. Thus, in this paper we introduce new tools for proving superrigidity theorems in rank one with target groups that are not compatible with $G$. While these new tools can be used for $\SO_0(n,1)$, see for instance Remark \ref{rem:SOremark}, our focus will be to apply these ideas in the case $G = \SU(n,1)$ to prove Theorem \ref{thm:main} and the other applications described later in this introduction.

It is well-known that if $G$ is a connected adjoint semisimple Lie group with no compact factors and $\Gam < G$ is an irreducible nonarithmetic lattice, then $G$ is necessarily isomorphic to either $\PO_0(n,1)$ or $\PU(n,1)$ for some $n \ge 2$. Therefore, combining \cite[Thm.\ 1.1]{BFMS} and Theorem~\ref{thm:main} with the fact that geodesic subspaces of arithmetic manifolds are arithmetic yields:

\begin{cor}[Finiteness]\label{cor:main}
Let $M$ be a nonarithmetic finite volume irreducible locally symmetric space of noncompact type. Then
\begin{enumerate}
\item $M$ contains only finitely many maximal geodesic subspaces of real dimension at least $2$, and
\item $M$ contains only finitely many nonarithmetic geodesic subspaces of real dimension at least $2$.
\end{enumerate}
\end{cor}

Note that a complex hyperbolic $n$-manifold can contain geodesic submanifolds of real dimension $k$ that are real hyperbolic ($2 \le k \le n$) or complex hyperbolic ($2 \le k \le 2 n - 2$ even), see \S\ref{ssec:arithex} for examples. We also remark that there are currently only $24$ known commensurability classes of nonarithmetic complex hyperbolic manifolds of finite volume, twenty-two in complex dimension $2$ and two in complex dimension $3$ \cite{Deraux3d}. Finding more examples or any examples in complex dimension $4$ and higher is an important open problem, e.g., see \cite[Prob.\ 9]{MargulisPCR} or \cite[Conj.\ 2.6]{Kapovich}.

We also note that the current technology for building nonarithmetic examples is consistent with the strategies in the real hyperbolic setting in that it produces examples with special geodesic subspaces. In particular, all known nonarithmetic complex hyperbolic lattices are commensurable with complex reflection groups; see \cite{DPP} for a comprehensive discussion for $\PU(2,1)$. There are also some attempts to use a hybrid construction inspired by the work of Gromov and Piatetski-Shapiro \cite{GPS}, for example \cite[Conj.\ 2.7]{Kapovich} or \cite{PaupertWells}. In this direction Theorem~\ref{thm:main} can be viewed as a criterion for (non)arithmeticity that can be applied when one has intimate knowledge of the collection of geodesic submanifolds, and we will interpret this criterion in terms of the topology of $M$ as an algebraic variety in Theorem \ref{thm:mainAG} below.

\medskip

We now state our general superrigidity result. Part \eqref{thm:SRcomp} mirrors \cite[Thm.\ 1.6]{BFMS}. However, the tools developed in \cite{BFMS} cannot be used to prove part \eqref{thm:SRpun1}. The heart of this paper develops new tools for proving superrigidity and applies them to prove part \eqref{thm:SRpun1}.

\begin{thm}[Superrigidity]\label{thm:SR}
Let $G$ be $\SU(n,1)$ for $n \ge 2$, $W < G$ be a noncompact connected almost simple subgroup, and $\Gam < G$ be a lattice. Suppose that $k$ is a local field, $\bfH$ is a connected adjoint $k$-algebraic group, and $\rho : \Gam \to \bfH(k)$ is a homomorphism with unbounded, Zariski dense image. Moreover, suppose that there is a faithful irreducible representation of $\bfH(k)$ on a $k$-vector space $V$ of dimension at least two and a $W$-invariant, ergodic measure $\nu$ on $(G \times \bbP(V)) / \Gam$ that projects to Haar measure on $G / \Gam$. If either
\begin{enumerate}

\item the pair consisting of $k$ and $\bfH$ is compatible with $G$, or \label{thm:SRcomp}

\item $k = \bbR$ and $\bfH(\bbR) \cong \PU(n,1)$, \label{thm:SRpun1}

\end{enumerate}
then $\rho$ extends to a continuous homomorphism from $G$ to $\bfH(k)$.
\end{thm}

We briefly recall the definition of compatibility from \cite{BFMS}. Let $P$ be a minimal parabolic subgroup of $G$ and $U$ its unipotent radical (see \S\ref{subsec:subgroups}). A pair consisting of a local field $k$ and a $k$-algebraic group $\bfH$ is said to be \emph{compatible} with $G$ if for every nontrivial $k$-subgroup $\bfJ<\bH$ and any continuous homomorphism $\tau:P \rightarrow N_{\bfH}(\bfJ)/\bfJ(k)$, where $N_{\bfH}(\bfJ)$ is the normalizer of $\bfJ$ in $\bfH$, we have that the Zariski closure of $\tau(U^\prime)$ coincides with the Zariski closure of $\tau(U)$ for every nontrivial subgroup $U^\prime<U$. Note that minimal parabolic subgroups are unique up to conjugation, so this is a property of $G$, not just $P$. We will extend this definition and elaborate on it in \S\ref{ssec:Compatible}.

\begin{rem}
Margulis asked when superrigidity holds for arithmetic complex hyperbolic lattices \cite[Prob.\ 9]{MargulisPCR}, and our results provide a partial answer to his question. Indeed, Theorem \ref{thm:SR} proves superrigidity of certain representations of arithmetic complex hyperbolic lattices, and we will describe applications of this later in the introduction. We also note that there are a number of previous superrigidity results for particular representations of complex hyperbolic lattices. For example, see the famous work of Toledo \cite{Toledo} and Corlette \cite{Corlette} on what are now known as \emph{maximal representations}, and see more recent work of Burger--Iozzi \cite{BurgerIozzi}, Pozzetti \cite{Pozzetti}, and Koziarz--Maubon \cite{KoziarzMaubon} for more on results of this kind and historical remarks. Another famous example is the proof by Klingler of superrigidity of low-dimensional representations of fundamental groups of fake projective planes, which he then used to deduce their arithmeticity \cite{Klingler}; see \cite{Klingler2} for a more general result along these lines.
\end{rem}

We now describe the general ideas behind the proof of Theorem~\ref{thm:SR}\eqref{thm:SRpun1}. The failure of $\PU(n,1)$ to be compatible with $\SU(n,1)$ can be measured precisely, and we describe this in \S\ref{ssec:Compatible}.
Specifically, we define an \emph{incompatibility datum} for a group $G$ with respect to an algebraic group $\bfH$ over a local field $k$ as a measure of the failure of compatibility. When $G$ is a real semisimple group, Proposition~\ref{prop:compk} shows that the only relevant case is when $k$ is $\bbR$ or $\bbC$ and Lemma~\ref{lem:Spar} proves that an incompatibility datum for the pair $(k, \bfH)$ is associated with a particular parabolic subgroup $Q$ of $\bfH(k)$.

When $G = \SU(n,1)$, $k = \bbR$, and $\bfH(k) = \PU(n,1)$, we use this general setup to produce a measurable self-map of the boundary of complex hyperbolic space having certain equivariance properties. A delicate fiber product argument translates the original map into an incidence geometry problem about \emph{chains} on the boundary of complex hyperbolic space. This problem was solved by Pozzetti \cite[Thm.\ 1.6]{Pozzetti} (see Theorem~\ref{thm:Pozzetti} for a statement in our language), and this allows us to complete the proof of Theorem~\ref{thm:SR}\eqref{thm:SRpun1}. More generally, the tools developed in \S\ref{ssec:Compatible} can be used to turn some superrigidity questions into problems about incidence geometry on boundaries. This is reminiscent of both Mostow's proof of strong rigidity in higher rank and the proof by Margulis and Mohammadi of the analogue of Theorem \ref{thm:main} for compact hyperbolic $3$-manifolds \cite{MM} (see Remark \ref{rem:SOremark}).

\medskip

The basic strategy for proving Theorem~\ref{thm:main} is to employ the general outline for Margulis's proof of arithmeticity of higher rank lattices \cite{MargulisSuperrigidity, MargulisBook}. This is the same strategy used in \cite{BFMS} for the real hyperbolic case; see \S\ref{ssec:arithpfsetup} for discussion. However, we need more than Theorem~\ref{thm:SR} to prove Theorem~\ref{thm:main} and we now describe the additional input we require.

Suppose that $M$ is a finite volume complex hyperbolic $n$-manifold, $n \ge 2$, and that $\Gam < \SU(n,1)$ is a lattice so that $M$ is biholomorphic to $\Gam \bs \bbB^n$. Canonically associated with $\Gam$ is its \emph{adjoint trace field} $\ell$, and the Zariski closure of $\Gam$ under the adjoint representation is an adjoint simple $\ell$-algebraic group $\conj{\bfG}$. The projection $\Gam \to \PU(n,1)$ is determined by a real place $v_0$ of $\ell$ so that $\conj{\bfG}(\ell_{v_0}) \cong \PU(n,1)$, where $\ell_v$ denotes the completion of $\ell$ at a place $v$. Then $\Gam$ is arithmetic if and only if $\Gam$ is precompact in $\conj{\bfG}(\ell_v)$ for all places $v \neq v_0$ of $\ell$.

In proving the analogue of Theorem~\ref{thm:main} in the real hyperbolic setting, the pairs $\ell_v$ and $\conj{\bfG}$ one encounters are always compatible with $\SO_0(n, 1)$ and so the analogue of Theorem~\ref{thm:SR}\eqref{thm:SRcomp} is all we needed. Certain pairs of interest to us in the proof of Theorem~\ref{thm:main}, particularly $(\bbR, \PU(r,s))$ with $2 \le s \le r < n$ and $r+s = n+1$, are compatible with $\SU(n,1)$. Theorem~\ref{thm:SR}\eqref{thm:SRpun1} takes care of the additional case where $\ell_v = \bbR$ and $\conj{\bfG}(\ell_v) \cong \PU(n,1)$. However, the above is not sufficient to handle the pair $(\bbC, \SL_{n+1}(\bbC))$.

To finish the proof of Theorem~\ref{thm:main}, we apply Simpson's celebrated results on linear representations of K\"ahler groups \cite{Simpson} and their generalization to the quasi-projective case. These methods actually lead to a more general theorem about the possibilities for the adjoint trace field $\ell$ of a lattice $\Gam$ in $\SU(n,1)$ and the associated $\ell$-algebraic group, namely the $\ell$-Zariski closure of $\Gam$ under the adjoint representation (see \S\ref{sec:Hodge} for the precise definition). In \S\ref{sec:Hodge} we will prove the following, which is known to experts but is not in the literature.

\begin{thm}[Hodge type and integral] \label{thm:SRht}
Let $\Gam < \SU(n,1)$ be a lattice, $\ell$ be its adjoint trace field, and $\bfG$ be the absolutely almost simple simply connected $\ell$-algebraic group canonically associated with $\Gam$. Then:
\begin{enumerate}

\item $\ell$ is totally real and the quadratic extension $\ell^\prime / \ell$ associated with $\bfG$ as a group of type ${}^{2}\mathrm{A}_n$ is totally complex;

\item for each real place $v$ of $\ell$, $\bfG(\ell \otimes_v \bbR)$ is isomorphic to $\SU(r_v,s_v)$ for some $r_v, s_v \ge 0$ with $r_v+s_v = n+1$;

\item $\Gam$ is integral, i.e., there is an action of $\bfG(\ell)$ on an $\ell$-vector space $V$ and an $\calO_\ell$-lattice $\calL \subset V$ so that a finite index subgroup of $\Gam$ stabilizes $\calL$.
\end{enumerate}
In particular, if $v_0$ is the place of $\ell$ associated with the lattice embedding of $\Gam$ in $\SU(n,1)$, then $\Gam$ is arithmetic if and only if $\bfG(\ell_v) \cong \SU(n+1)$ for all archimedean places $v \neq v_0$ of $\ell$.
\end{thm}

The first two statements and their proofs were described to us by Domingo Toledo, and the third follows from work of Esnault--Groechenig \cite{EsnaultGroechenig}. To prove Theorem~\ref{thm:main}, we only need parts (1) and (2) of Theorem~\ref{thm:SRht}, and in the cocompact case these are immediate from Simpson's result that rigid representations of K\"ahler groups are of Hodge type \cite[Lem.\ 4.5]{Simpson}. There is in fact considerable overlap between the cases in Theorem~\ref{thm:main} covered by Theorems~\ref{thm:SR}\eqref{thm:SRcomp} and~\ref{thm:SRht}. Theorem~\ref{thm:SRht}(3) rules out the cases where $\ell_v$ is nonarchimedean, but these cases are perhaps more easily handled by Theorem~\ref{thm:SR}\eqref{thm:SRcomp}, since any simple algebraic group over a nonarchimedean local field is compatible with $G$ by Proposition~\ref{prop:compk}.

We now describe some other applications and interpretations of our results in the language of algebraic and complex geometry.

\medskip

One application of Theorem~\ref{thm:main} is to provide evidence for a conjecture of Klingler \cite[Conj.\ 1.12]{KlinglerSurvey}. Let $M = \bbB^n / \Gam$ be a finite volume complex hyperbolic manifold of complex dimension $n \geq 2$. As discussed above, there exists a totally real number field with fixed embedding $v_0: \ell \to \R$ and an almost simple $\ell$-algebraic group $\bfG$ with $\bfG(\ell \otimes_{v_0} \R) \cong \SU(n,1)$ such that, up to passing to a subgroup of finite index, $\Gamma \subseteq \bfG(\mathcal{O}_{\ell})$, where $\mathcal{O}_{\ell}$ denotes the ring of integers of $\ell$. Let $\bfH$ denote the Weil restriction of scalars from $\ell$ to $\Q$ of $\bfG$, which is a semisimple $\Q$-group with $ \Gamma \subseteq \bfH(\Z)$. Any faithful representation $\rho: \bfH \rightarrow \GL(V)$ defined over $\Z$ induces a polarizable $\Z$-variation of Hodge structure $\mathbb{V}$ on $M$.

As soon as $\Gamma$ is nonarithmetic, the group $\bfH(\ell \otimes_\bbQ \R)$ admits at least two noncompact factors and one easily checks that any totally geodesic subvariety of $M$ is atypical for $(M, \mathbb{V})$ in the sense of \cite[Def.\ 1.7]{KlinglerSurvey}. One can in fact check that maximal totally geodesic subvarieties of $M$ are optimal in the sense of \cite[Def.\ 1.8]{KlinglerSurvey}. Then \cite[Conj.\ 1.12]{KlinglerSurvey} implies that $M$ contains at most finitely many maximal totally geodesic subvarieties. Therefore, our results confirm this consequence of \cite[Conj.\ 1.12]{KlinglerSurvey}. See \cite{BaldiUllmo} for more about these connections.

\medskip

An application in another direction is to Margulis's question as to whether arithmeticity can be detected at the topological level. We explain in \S\ref{ssec:AG} how Theorem~\ref{thm:main} can be interpreted as saying that arithmeticity is detected by the topology of the complex variety underlying a complex hyperbolic manifold. In fact, arithmeticity is recognized by the structure of the intersection product on cohomology. For simplicity, we state here our result for complex dimension $2$ and refer to \S\ref{ssec:AG} for more on the higher dimensional case.

Suppose $X$ is a smooth complex projective variety with canonical divisor $K_X$ and $D \subset X$ is a (possibly empty) smooth divisor such that  $M = X \ssm D$ is a complex hyperbolic $2$-manifold for which $X$ is a smooth toroidal compactification of $M$. We recall that any complex projective curve $C$ on $X$ satisfies
\begin{equation}\label{eq:Proportionality}
3\, C \cdot C + 3 \deg(D \cap C) \ge - K_X \cdot C + 2\, D \cdot C,
\end{equation}
with respect to the intersection pairing on $H^2(X)$, and equality holds if and only if $C$ is totally geodesic. Note that $D$ is empty if and only if $M$ is compact, where Equation \eqref{eq:Proportionality} reduces to
\[
3\, C \cdot C \ge - K_X \cdot C.
\]
We also note that every complex hyperbolic manifold admits a finite covering with a compactification as assumed in Theorem~\ref{thm:mainAG}; see the proof of Theorem~\ref{thm:SRht} for a precise discussion. Therefore, up to finite covers there is no loss of generality in making the assumptions we make in the following result.

\begin{thm}[Arithmeticity and the intersection pairing]\label{thm:mainAG}
Suppose that $X$ is the smooth toroidal compactification of a complex hyperbolic $2$-manifold with (possibly empty) compactification divisor $D$. If $X$ contains infinitely many complex projective curves $C$ where equality holds in \eqref{eq:Proportionality} then $M$ is arithmetic. Equivalently, if $M$ is nonarithmetic, then there are only finitely many curves $C$ on $X$ for which equality holds in \eqref{eq:Proportionality}.
\end{thm}

\medskip

Another application of Theorem~\ref{thm:SR} is to ruling out existence of certain mappings between complex hyperbolic manifolds. For example, Siu asked whether or not there are surjective holomorphic mappings $M_1 \to M_2$ between ball quotients with $2 \le \dim_\C(M_2) < \dim_\C(M_1)$ \cite{SiuConj}. The only progress thus far is by Koziarz and Mok, who ruled out the case of holomorphic submersions \cite{KoziarzMok}. Using Theorem~\ref{thm:SR} we prove the following in \S\ref{ssec:Siu}.

\begin{thm}[Nonexistence of certain maps]\label{thm:Siu}
Let $M$ be a finite volume complex hyperbolic manifold of complex dimension $n \ge 2$ containing a family $\{Z_i\}$ of geodesic submanifolds, all of real dimension at least two, that are equidistributed in the sense of Proposition~\ref{prop:BFMS3.1}. Suppose that $N$ is a finite volume complex hyperbolic manifold $N$ with $\dim_\bbC(N) = d \le n$ and $f : M \to N$ is a continuous map such that $f(Z_i)$ is contained in a proper geodesic submanifold of $N$ for all $i$ and
\[
f_*(\pi_1(M)) \le \pi_1(N) < \PU(m,1)
\]
is Zariski dense. Then $d = n$ and $f$ is homotopic to a finite cover.
\end{thm}

An example application of Theorem~\ref{thm:Siu} is to the maps between Deligne--Mostow orbifolds, described in detail in work of Toledo \cite[Thm.~1]{ToledoMaps}. There are surjective holomorphic maps $f : M \to N$ between complex hyperbolic $2$-orbifolds that are not homotopic to a finite cover where $M$ is arithmetic with a family $\{Z_i\}$ as in Theorem~\ref{thm:Siu}. The maps are additionally surjective on the level of orbifold fundamental groups, hence Theorem~\ref{thm:Siu} implies that $f_*(\pi_1(Z_i))$ must be Zariski dense in $\PU(2,1)$ for all but finitely many $i$. However, note that it is possible that some finite number of $Z_i$ are contracted by the map, hence $f_*(\pi_1(Z_i))$ is trivial for those $i$.

Our last application is the following. There are lattices $\Gam < \SL_2(\bbC)$ such that there are $\gam \in \Gam$ so that $\Tr(\Ad(\gam))$ is not an algebraic integer. See \cite[\S 5.2.2]{MaclachlanReid} and \cite{Cheseblo} for many examples. In \S\ref{ssec:NotTG} we will show that part $(3)$ of Theorem~\ref{thm:SRht} has the following consequence pointed out to us by Ian Agol.

\begin{thm}[Restricting geodesic submanifolds]\label{thm:notTG}
There are both closed and noncompact finite volume hyperbolic $3$-manifolds that are not isometric to an immersed totally geodesic submanifold of a complex hyperbolic $n$-manifold for any $n$.
\end{thm}

We now discuss the organization of this paper. Section \ref{sec:sun1} starts by discussing $\SU(n,1)$ and some of its subgroups. In \S\ref{ssec:Compatible} we introduce the notion of an incompatibility datum, define compatibility, and prove some general results. We also prove compatibility of $\SU(n,1)$ with certain algebraic groups over local fields and study the failure of $\PU(n,1)$ to be compatible with $\SU(n,1)$. In \S\ref{sec:map} we discuss existence of equivariant maps and algebraic representations. In \S\ref{sec:SRcomp} we make the necessary modifications to the arguments in \cite{BFMS} to prove Theorem~\ref{thm:SR}\eqref{thm:SRcomp}. We continue in \S\ref{sec:SRpun1} with the proof of our main technical result, Theorem~\ref{thm:SR}\eqref{thm:SRpun1}. In \S\ref{sec:Hodge}, we prove Theorem~\ref{thm:SRht} and give some additional algebraic setup for the proof of Theorem~\ref{thm:main}, which is contained in \S\ref{sec:mainpf}. Then \S\ref{sec:Final} contains examples that exhibit the possible behavior of geodesic submanifolds of arithmetic complex hyperbolic manifolds, the proofs Theorems \ref{thm:mainAG}, \ref{thm:Siu}, and \ref{thm:notTG}, and some final questions.

\subsubsection*{Acknowledgments}
The authors thank Domingo Toledo for describing how Simpson's work leads to Theorem~\ref{thm:SRht} and Olivier Biquard for correspondence about the nonuniform case. We thank Bruno Klingler for explaining the relationship to his conjecture discussed in the introduction, Ian Agol for pointing out Theorem \ref{thm:notTG} and Venkataramana, Nicolas Bergeron, and Emmanuel Ullmo for insightful conversations. We also thank Jean L\'ecureux and Beatrice Pozzetti for conversations related to the proof of Theorem~\ref{thm:SR}\eqref{thm:SRpun1}.
It is a pleasure to thank Baldi and Ullmo for having explained their approach to us and for various conversations around Hodge Theory. We are grateful to the referees for the constructive comments and recommendations which improved the quality of our exposition.

Bader was supported by the ISF Moked 713510 grant number 2919/19. Fisher was supported by NSF DMS-1906107. Miller was supported by  NSF DMS-2005438/2300370. Stover was supported by Grant Number 523197 from the Simons Foundation/SFARI and Grants DMS-1906088, DMS-2203555 from the National Science Foundation.

\section{Preliminaries on $\SU(n,1)$}\label{sec:sun1}

Throughout this section and the remainder of the paper we fix the group $G\cong \SU(n,1)$ with $n \ge 2$. We will consider $G$ as being the group of real points of a real algebraic group.

\subsection{The group $G$ and its standard subgroups.} \label{subsec:simple}

In this subsection we fix some notation and give convenient coordinates for certain subgroups of $G$ that will be important in what follows.

\medskip

We start by giving a convenient matrix representation for $G$. The group $\U(n,1)$ is often described as the automorphism group of the hermitian form
\[
h_0(x_1,\dots,x_{n+1}) = \sum_{i=1}^n |x_i|^2 - |x_{n+1}|^2,
\]
on $\bbC^{n+1}$.
Under the linear change of variables
\[
y_1=\frac{\sqrt{2}}{2}(x_1+x_{n+1}),~y_2=x_2,~\ldots, ~y_n=x_n,~y_{n+1}=\frac{\sqrt{2}}{2}(x_1-x_{n+1})
\]
this form becomes
\begin{equation} \label{eq:h}
h(y_1,\dots,y_{n+1}) = y_1\conj{y}_{n+1}+y_{n+1}\conj{y}_1+\sum_{i=2}^n  |y_i|^2.
\end{equation}
Hereafter, we will view $G$ as the subgroup of $\SL_{n+1}(\bbC)$ preserving $h$.

\begin{rem}
For $n=1$, the stabilizer of $h$ in $\SL_2(\mathbb{C})$ is the image of the homomorphism:
\begin{align*}
\SL_2(\mathbb{R}) &\to \SL_2(\mathbb{C}) \\
\begin{pmatrix}
a & b \\
c & d
\end{pmatrix}
&\mapsto
\begin{pmatrix}
a & ib \\
-ic & d
\end{pmatrix}
\end{align*}
This explicitly realizes the well-known isomorphism $\SU(1,1) \cong \SL_2(\mathbb{R})$. We will tacitly assume from here forward that $n\geq 2$.
\end{rem}

In what follows, let $e_i$ denote the $i^{th}$ standard basis vector in $\mathbb{C}^{n+1}$ given by $y_i = 1$ and $y_j=0$ for $j\neq i$. Note that the restriction of $h$ to the complex line spanned by the vector $e_1-e_{n+1}$ is negative definite. We denote by $K\le G$ the stabilizer in $G$ of this line. Then $K$ also stabilizes the $h$-orthogonal complement of this line, namely the complex hyperplane spanned by the vectors $e_2,\ldots,e_n$ and $e_1+e_{n+1}$, and note that $h$ is positive definite on this hyperplane. One sees easily that
\[
K\cong \mathrm{S}(\U(n) \times \U(1))\simeq \U(n),
\]
and that it is a maximal compact subgroup of $G$. In particular, every compact subgroup of $G$ is conjugate to a subgroup of $K$.

For $1 \le m \le n$ let $W^c_m\le G$ be the subgroup of $G$ fixing each of the $n-m$ standard basis vectors $e_{m+1},\ldots, e_n\in \mathbb{C}^{n+1}$, and note that $W^c_m\cong \SU(m,1)$. We also consider the subgroup
\[
W^c_m\cap \GL_{n+1}(\bbR) < \GL_{n+1}(\bbC),
\]
which is isomorphic to $\SO(m,1)$ and we denote its identity component by $W^r_m$, thus $W^r_m \cong \SO_0(m,1)$.

\begin{defn} \label{def:standard}
The subgroups $W^c_1,\ldots,W^c_n$ and $W^r_2,\ldots,W^r_n$ of $G$ defined above are said to be \emph{the standard almost simple subgroups}, or for short just \emph{the standard subgroups} of $G$.
\end{defn}

Identifying $K\bs G$ with complex hyperbolic $n$-space $\bbB^n$, the standard subgroups of $G$ have a special relationship with its totally geodesic subspaces. It is shown in \S3.1 of Goldman's book \cite{Goldman} (see in particular \S3.1.11) that every totally geodesic subspace of complex hyperbolic $n$-space with real dimension at least two is isometric to either real hyperbolic $m$-space for some $2 \le m \le n$ or complex hyperbolic $m$-space for some $1 \le m \le n$.

Moreover, the group $G$ acts transitively on the collection of totally geodesic subspaces of any given \emph{type}, where the type of a totally geodesic subspace describes whether it is isometric to real or complex hyperbolic space of a given fixed dimension. We note that a real hyperbolic $2$-plane has a different type from a complex hyperbolic line. Indeed, under the restriction of the complex hyperbolic metric the former has a constant sectional curvature $-1/4$ while the latter has curvature $-1$.

It is evident that for every standard subgroup $W\le G$, $W\cap K$ is a maximal compact subgroup of $W$ and that
\[
(W \cap K) \bs W \overset{\sim}{\longrightarrow} K \bs K W \subseteq K \bs G,
\]
is a totally geodesic embedding of a real or complex hyperbolic space of the corresponding type in $K\bs G$. The following proposition summarizes the above discussion.

\begin{prop}[{\cite[\S3.1]{Goldman}}] \label{prop:standard}
The totally geodesic subspaces of (real) dimension at least two in the symmetric space $K\bs G$ are exactly the subsets of the form $K\bs KWg$ for an element $g\in G$ and a standard subgroup $W\le G$ corresponding to the type of the given totally geodesic subspace.
\end{prop}

Proposition~\ref{prop:standard} will play a prominent role in \S\ref{ssec:Equi},
along with Lemma~\ref{lem:AlignedToTG} below in which we will record some basic facts about standard subgroups. First we state an important corollary of Proposition~\ref{prop:standard} regarding the classification of almost simple subgroups of $G$, which follows from the fact that every noncompact, almost simple subgroup stabilizes a totally geodesic subspace constructed by inclusion of maximal compact subgroups.

\begin{prop} \label{prop:almostsimple}
The standard subgroups of $G$ are noncompact, connected, almost simple, closed subgroups generated by unipotent elements, and every noncompact, connected, almost simple, closed subgroup of $G$ generated by unipotent elements is conjugate to a unique standard subgroup of $G$.
\end{prop}

\begin{defn} \label{def:type}
Given a noncompact, connected, almost simple, closed subgroup of $G$ generated by unipotent elements, we say that its \emph{type} is the type of the unique standard subgroup of $G$ to which it is conjugate.
\end{defn}

\begin{rem} \label{rem:almostsimple}
The analogue of Proposition~\ref{prop:almostsimple} for $\SO_0(n,1)$ holds as well: every noncompact, connected, almost simple, closed subgroup of $\SO_0(n,1)$ is conjugate to a unique standard subgroup of the form $\SO_0(m,1)$ for some $2\leq m \leq n$. The proof of this fact is similar to the proof of Proposition~\ref{prop:almostsimple} to be presented below,
only it is easier and well-known, so we will omit it.
\end{rem}

The proof of Proposition~\ref{prop:almostsimple} will be derived simultaneously with the proof of Lemma~\ref{lem:AlignedToTG} given below. For a subgroup $S\le G$, we use the notation $S^+$ to denote the closed subgroup of $G$ generated by all the one-dimensional unipotent subgroups of $S$. We note that $S^+$ is necessarily connected, and it is either a unipotent subgroup or a noncompact, almost simple subgroup of $G$. This follows from the fact that $G$ has rank 1.
Indeed, if $S^+$ is contained in a parabolic subgroup $P$ then it is contained in $P^+$, which is the unipotent radical of $P$, hence $S^+$ is unipotent. Otherwise, $S^+$ has a trivial unipotent radical, hence it is semisimple, and since it has no compact factor, it follows that it is almost simple and noncompact.

\begin{lemma}\label{lem:AlignedToTG}
Fix a standard subgroup $W\le G$ and let $N\le G$ be its normalizer. Then the following results hold.
\begin{enumerate}

\item
If $S\le G$ is a connected, almost simple, closed subgroup generated by unipotent elements that preserves $K \bs K W$ and acts transitively on it, then $S=W$.
\label{lem:AlignedToTG1}

\item
The stabilizer in $G$ of the totally geodesic subspace $K \bs K W$ is $N$. Moreover, $N\subseteq KW$ thus $N/W$ is compact and $N^+=W$.
\label{lem:AlignedToTG2}

\item Assume that $S \le G$ is a closed intermediate subgroup $W \le S\le N$. Then $S^+=W$ and $K\bs KW=K\bs K S=K\bs K N$.
\label{lem:AlignedToTG4}

\item Let $S \le G$ be a closed subgroup containing $W$. Then there exists $k\in K$ and a standard subgroup $W_0 \le G$ such that $kS^+k^{-1}=W_0$ and
\[
K\bs KS=K\bs KS^+=K\bs KW_0k.
\]
Moreover, $K \bs KS$ is a totally geodesic subspace of $K\bs G$ whose volume measure coincides, up to normalization, with the push-forward of Haar measure on  $S$ to $(S\cap K)\bs S\simeq K\bs KS$.
\label{lem:AlignedToTG3}

\end{enumerate}
\end{lemma}

\begin{proof}[Proofs of Proposition~\ref{prop:almostsimple} and Lemma~\ref{lem:AlignedToTG}]
We first prove part \eqref{lem:AlignedToTG1} of the lemma. Let $N_0$ be the stabilizer in $G$ of $K \bs K W$ and observe that $S$ is a subgroup of $N_0$ and  $S=S^+\le N_0^+$. Since it contains the almost simple subgroup $S$, $N_0^+$ is not unipotent, thus it is connected and almost simple. We will show that in fact $S=N_0^+$.

Since $N_0$ acts by isometries on the symmetric space $Z=W\cap K \bs W \simeq K \bs K W$ we have a natural continuous homomorphism from $N_0$ to the group of isometries, $\Isom(Z)$, endowed with the compact open topology. It goes back to Elie Cartan that the image of $S$ coincides with the identity component of $\Isom(Z)$, as $S$ is semisimple connected group of isometries of the symmetric space $Z$ \cite[Thm.~V.4.1]{Helgason}. The same holds for $N_0^+$. Since both $S$ and $N_0^+$ are almost simple, the restriction of the above homomorphism to each has a finite kernel. It follows that $S$ has finite index in $N_0^+$, hence indeed $S=N_0^+$, since $N_0^+$ is connected.

In the above discussion $S$ was arbitrary, thus applying it in the special case $S=W$ we conclude that $W=N_0^+$. Therefore $S=W$ holds \emph{a priori} and this proves part \eqref{lem:AlignedToTG1}.

Next we prove part \eqref{lem:AlignedToTG2} of the lemma. By the previous discussion we have that $N_0\leq N$, since $W=N_0^+$ is normal in $N_0$. Observe that $K \bs K W$ is the unique $W$-invariant totally geodesic subspace of its type. Indeed, given any $W$-invariant totally geodesic subspace of the same type $Z\subseteq K\bs G$ and $z \in Z$, the function $d(\cdot,zW)$ measuring distance to the $W$-orbit of $z$ is constant on $K \bs K W$. This implies that $z W$ and $K \bs K W$ have identical boundaries in the visual compactification of $K\bs G$. Since $K \bs G$ has negative curvature, this implies that $z W \subset K \bs K W$, hence $Z = K \bs K W$.

It also follows that $N\leq N_0$, so $N=N_0$ and $N$ is indeed the stabilizer in $G$ of the totally geodesic subspace $K \bs K W$. We furthermore see that $N^+=N_0^+=W$. For every $n\in N$, it also follows that $K\bs Kn=K\bs Kw$ for some $w\in W$, thus $N\subseteq KW$. This proves part \eqref{lem:AlignedToTG2}.

Part \eqref{lem:AlignedToTG4} of the lemma follows immediately from part \eqref{lem:AlignedToTG2}. Indeed, notice that $W=W^+\leq S^+ \leq N^+=W$ implies that $S^+=W$ and the sequence of inclusions
\[
K\bs KW \subseteq K\bs K S \subseteq K\bs K N \subseteq K\bs K(KW) =  K\bs KW
\]
then has equality everywhere.

We now turn to the proof of Proposition~\ref{prop:almostsimple}. The fact that the standard subgroups are pairwise nonconjugate, noncompact, connected, almost simple, closed subgroups of $G$ generated by unipotent elements is obvious, so we only need to show that any other such group is conjugate to a standard one. Let $S \le G$ be a noncompact, connected, almost simple, closed subgroup generated by unipotent elements. By the Karpelevich--Mostow Theorem \cite{Karpel,MostowDec} for some $h \in G$ the $S$-orbit $K\bs KhS \subseteq K\bs G$ is totally geodesic. Thus, by Proposition~\ref{prop:standard}, there exist an element $g\in G$ and a standard subgroup $W_0\le G$ such that $K\bs KhS=K\bs KW_0g$. Rewriting we have $K\bs Khg^{-1}S^g=K\bs KW_0$ and we conclude that $S^g$ preserves $K\bs KW_0$ and acts transitively on it. By part \eqref{lem:AlignedToTG1} we get that $S^g=W_0$ and this proves Proposition~\ref{prop:almostsimple}.

We are now in a position to prove part \eqref{lem:AlignedToTG3} of the lemma. Note that $W\le S$ implies that $W=W^+ \le S^+$. It follows that $S^+$ is not unipotent, thus it is connected and almost simple. Using Proposition~\ref{prop:almostsimple} there exist an element $g\in G$ and a standard subgroup $W_0\le G$ such that $(S^+)^g=W_0$. We conclude that $W^g\le W_0$. We claim that the group $W^g$, which is a noncompact, connected, almost simple, closed subgroup of $W_0$, is conjugate in $W_0$ to a standard subgroup of $W_0$. Indeed, when $W_0=W^c_m$ for some $1\le m\le n$ this follows from Proposition~\ref{prop:almostsimple}, replacing the role of $G$ by $W_0\cong \SU(m,1)$, and when $W_0=W^r_m$ for some $2\le m\le n$ this follows from Remark~\ref{rem:almostsimple} applied to $W_0\simeq \SO_0(m,1)$.
In any case, this standard subgroup of $W_0$ must be of the same type as $W^g$, so it must be $W$ itself.

Therefore, there is an $h\in W_0$ such that $W^{hg}=(W^g)^h=W$. It follows that $hg\in N$, thus by part \eqref{lem:AlignedToTG2}, $hg=kw$ for some $k\in K$ and $w\in W$. The sequence of equations
\[
kS^+k^{-1}=kwS^+w^{-1}k^{-1}=hgS^+g^{-1}h^{-1}=hW_0h^{-1}=W_0,
\]
implies that $(kSk^{-1})^+=kS^+k^{-1}=W_0$, and therefore
\[
W_0\le kSk^{-1} \le N_0,
\]
where $N_0$ is the normalizer of $W_0$ in $G$. By part \eqref{lem:AlignedToTG4}
we get
\[
K\bs K S^+k^{-1}=K\bs K (S^+)^k=K\bs KW_0=K\bs K S^k=K\bs K Sk^{-1},
\]
and upon applying $k$ on the right we conclude that indeed
\[
K\bs KS=K\bs KS^+=K\bs KW_0k.
\]
This is a totally geodesic subspace of $K\bs G$ by Proposition~\ref{prop:standard}. The final statement follows from the essential uniqueness of an $S$-invariant measure on $(S\cap K) \bs S \simeq K\bs K S$.
This proves part \eqref{lem:AlignedToTG3} and thus completes the proof.
\end{proof}

\subsection{Parabolic subgroups of $G$ and some of their subgroups}\label{subsec:subgroups}

In this subsection we continue the discussion of special subgroups of $G$ begun in the previous one. See \cite[\S VII.7]{Knapp} for definitions and properties of the structure theory for simple Lie groups and their parabolic subgroups. Now we focus mostly on parabolic subgroups of $G$ and some of their subgroups. Recall that $e_1,\ldots,e_{n+1}$ denotes the standard basis of $\mathbb{C}^{n+1}$ and that $G$ is the subgroup of $\SL_{n+1}(\bbC)$ preserving the form $h$ given in Equation \eqref{eq:h}.

Let $P<G$ be the stabilizer of the isotropic line spanned by $e_1$. This is a parabolic subgroup of $G$ and all proper parabolic subgroups of $G$ are conjugate to $P$, since $G$ has real rank $1$. Some important subgroups of $P$ are
\begin{align*}
A & = \left\{
\left( \begin{array}{c|ccc|c}
\lambda & & 0 & & 0 \\
\hline
~ & & & & ~ \\
0 & & I_{n-1} & & 0 \\
~ & & & & ~ \\
\hline
0 & & 0 & & \lambda^{-1}
\end{array} \right)
\ : \ \lambda \in \mathbb{R}^* \right\}, \\
M & = \left\{ \left(
\begin{array}{c|ccc|c}
\theta & & 0 & & 0 \\
\hline
~ & & & & ~ \\
0 & & T & & 0 \\
~ & & & & ~ \\
\hline
0 & & 0 & & \theta
\end{array} \right)
\ :\
\begin{array}{c}
\theta \in \U(1),\\
T \in \U(n-1) \\
\theta^2=\det(T)^{-1}
\end{array}
\right\}, \\
U & =\left\{ \left(
\begin{array}{c|ccc|c}
1 & & -v^* & & ib -\frac{1}{2}\|v\|^2 \\
\hline
~ & & & & ~ \\
0 & &  I_{n-1} & & v \\
~ & & & & ~ \\
\hline
0 & & 0 & & 1 \\
\end{array} \right)
\ :\  v\in \mathbb{C}^{n-1},~b\in \mathbb{R}
\right\}.
\end{align*}
Here $v^*$ denotes the complex conjugate transpose of $v$. Note that $U$ is isomorphic to the real $(2n-1)$-dimensional Heisenberg group $\bbH_{2n-1}(\mathbb{R})$ and that $A \cong \bbR^*$.

One checks that $P$ is generated by $M$, $A$, and $U$. Note that $A$ is a maximal $\bbR$-split torus of $G$, $U$ is the unipotent radical of $P$, and $M$ is a compact reductive group that commutes with $A$. Thus the Langlands decomposition of $P$ is given by $P=MAU$.

Every element of $P$ can be represented in an obvious way by parameters $(\lambda,\theta,T,v,b)$ using the coordinates introduced above. Such a representation is unique up to the order two intersection group $M\cap A$, that is, $(\lambda,\theta,T,v,b)$ represents the same element as $(-\lambda,-\theta,T,v,b)$. In the remainder of this subsection we will use this representation often. We will also represent elements of $M$, $A$, and $U$ by $(\theta,T)$, $\lambda$, and $(v,b)$ in the obvious way.

Let $C < M$ be the subgroup of scalar matrices in $G$. These scalars are the $(n+1)^{st}$ roots of unity, thus $C$ is a cyclic group of order $n+1$. Note that $C$ is the center of $G$, and it is easy to see that it is also the center of $P$.

The group $U$ is a two step nilpotent group with center
\[
Z =
\left\{ \left(
\begin{array}{c|ccc|c}
1 & & 0 & & ib \\
\hline
~ & & & & ~ \\
0 & & I_n & & 0 \\
~ & & & & ~ \\
\hline
0 & & 0 & & 1 \\
\end{array} \right)
\ :\  b\in \mathbb{R}
\right\}.
\]
This is a characteristic subgroup of $U$, hence it is a normal subgroup of $P$. Under the identification $Z\cong \mathbb{R}$, the conjugation action of $P$ on $Z$ is given by the homomorphism:
\begin{align*}
P &\to \mathbb{R}^*_+<\mathbb{R}^*\cong \GL_1(\mathbb{R}) \\
(\lambda,\theta,T,v,b) &\mapsto \lambda^2
\end{align*}
Since $\mathbb{R}^*_+$ has no nontrivial compact subgroups, we obtain the following, which we record for future reference.

\begin{lemma} \label{lem:Zcent}
Every compact subgroup of $P$ commutes with $Z$.
\end{lemma}

The quotient group $U/Z$ is naturally identified with $\mathbb{C}^{n-1}$
by the map $(v,b)\mapsto v$. Under this identification, the conjugation action of $P$ on $U/Z$ is given by the homomorphism:
\begin{align*}
P &\to \bbR^*\cdot \U(n-1)<\GL_{n-1}(\mathbb{C}) \\
(\lambda,\theta,T,v,b) &\mapsto \lambda \theta^{-1}T
\end{align*}
Here $\mathbb{R}^*<\GL_{n-1}(\mathbb{C})$ is considered as the group of real scalar matrices. Note that this homomorphism induces isomorphisms $M/C \cong  \U(n-1)$ and $MA/C \cong \bbR^*\cdot \U(n-1)$. In particular, the conjugation action of $MA/C$ on $(U / Z) \ssm \{0\}$ is transitive and faithful. This is crucial in proving the following lemma, which describes the normal subgroups of $P$.

\begin{lemma} \label{lem:normalinP}
If $N \trianglelefteq P$ is a normal subgroup, either $U \le N$ or $N \le CZ$.
\end{lemma}

\begin{proof}
Consider the quotient map $\theta:P\to P/CZ$. Then $\theta(N)\cap \theta(U)$ is normal in $P/CZ$, hence transitivity of the $MA/C$ conjugation action on $(U/Z) \ssm \{0\}$ implies that this intersection is either trivial or all of $\theta(U)$. Assume the latter case. We get that $UC \le NCZ$, thus $UC \le (UC \cap N)CZ$.
Taking commutators and using the centrality of $CZ$ in $CU$, we get
\begin{align*}
Z=[U,U]&=[UC,UC] \\
&\le [(UC\cap N)CZ,(UC\cap N)CZ] =[CU\cap N,CU\cap N]
\end{align*}
and we deduce that $Z\le N$. It follows that $UC \leq CN$, but $U$ is connected and $C N / N$ is totally disconnected, hence $U \le N$ and we are done.

Therefore, it remains to consider the case where $\theta(N)\cap \theta(U)$ is trivial. It follows that $\theta(N)$ commutes with $\theta(U)$, as they are both normal subgroups of $P/CZ$. Since the action of $\theta(MA) \cong MA/C$ on $\theta(U) \cong U/Z$ is faithful, we deduce that $N$ is in the kernel of the natural map $P =MAU \to MA/C$, that is $N\le UC$. By triviality of $\theta(N)\cap \theta(U)=\theta(N)\cap \theta(UC)$, we see that $N$ is in the kernel of
$\theta$, hence $N \le CZ$ as desired.
\end{proof}

Let $D<G$ be the subgroup stabilizing the plane spanned by $e_1$ and $e_{n+1}$. Thus $D$ also stabilizes the $h$-orthogonal complement of this plane, the subspace spanned by $e_2,\ldots,e_n$,
and we have:
\[
D = \left\{ \left(
\begin{array}{c|ccc|c}
\theta a & & 0 & & i\theta b \\
\hline
~ & & & & ~ \\
0 & &  T & & 0 \\
~ & & & & ~ \\
\hline
-i\theta c & & 0 & & \theta d
\end{array} \right)
\ :\
\begin{array}{c}
a,b,c,d \in \mathbb{R},\, ad-bc=1,\\
\theta \in \U(1),\, T\in \U(n-1),\\
\theta^2=\det(T)^{-1}
\end{array}
\right\}.
\]

We will be particularly interested in the group $P\cap D$ stabilizing both the line spanned by $e_1$ and the plane spanned by $e_1$ and $e_{n+1}$.
\begin{lemma} \label{lem:DcapP}
We have $P\cap D=MAZ$.
\end{lemma}

\begin{proof}
This is easily deduced from the matrix representations of $P$ and $D$.
\end{proof}

The following lemma is easy, nevertheless we give a detailed proof due to its importance to our considerations.

\begin{lemma} \label{lem:aut(P)}
Fix $d\geq 1$ and consider the group
\[ \bbR^*\cdot \U(d) \ltimes \mathbb{C}^{d} \leq \GL_{d}(\bbC) \ltimes \mathbb{C}^{d}. \]
Its group of outer automorphisms, $\Out(\bbR^*\cdot \U(d) \ltimes \mathbb{C}^{d})$, is of order 2 and its nontrivial element corresponds to complex conjugation of $\GL_{d}(\bbC) \ltimes \mathbb{C}^{d}$.
\end{lemma}

\begin{proof}
For brevity, we set $S=\bbR^*\cdot \U(d)$ and $V=\mathbb{C}^{d}$. We consider an automorphism $\tau:S\ltimes V\to S\ltimes V$ and will show that, up to an inner automorphism, $\tau$ is either trivial or complex conjugation.

Identifying $V$ with $\mathbb{R}^{2d}$ in the standard way, we identify its group of continuous automorphisms with the the real general linear group $\GL_{2d}(\mathbb{R})$. For an operator $t\in \GL_{2d}(\mathbb{R})$ and $v\in V$ we denote the corresponding action by $t\cdot v$. We identify the conjugation action of $S$ on $V$ as the linear action of $S<\GL_{2d}(\mathbb{R})$ and for $s\in S$ and $v\in V$, define $s\cdot v=svs^{-1}=v^s$.

We have that $\tau(S)<S\ltimes V$ is a Levi subgroup. See \cite[\S6.8]{Bourbaki1} for details about Levi decompositions (at the level of Lie algebras). As all Levi subgroups are conjugate, we assume as we may that $\tau(S)=S$. Let $\al$ denote the automorphism of $S$ induced by $\tau$. Since $V$ is a characteristic subgroup of $S\ltimes V$, we have $\tau(V)=V$. Thus $\tau$ induces a continuous automorphism of $V$ corresponding to a fixed element $t\in \GL_{2d}(\mathbb{R})$. For $s\in S$ and $v\in V$ we have
\[
t s\cdot v=t\cdot v^s=\tau(v^s)=\tau(v)^{\tau(s)}=(t\cdot v)^{\alpha(s)}=\alpha(s) t\cdot v.
\]
Thus $\alpha(s)=s^t$. We conclude that $t$ is contained in $N$, where $N$ is the normalizer of $S$ in $\GL_{2d}(\mathbb{R})$, and $\alpha$ is the automorphism of $S$ obtained by conjugating by $t$.

We claim that $S<N$ is of index two, where the nontrivial coset is generated by complex conjugation. We fix $g\in N$ and argue that, up to complex conjugation, $g$ is in $S$. Note that $N$ commutes with the group of real scalars $\mathbb{R}^*<S$ and the group $\U(1)<S$ consisting of modulus 1 complex scalars is characteristic. Indeed, $\U(d)<S$ is characteristic, as it is the unique maximal compact subgroup of $S$, and $\U(1)$ is its center. Thus $g$ normalizes $U(1)$.

Since the unique nontrivial automorphism of $\U(1)$ is complex conjugation, we can assume that $g$ acts trivially on $\U(1)$. We conclude that $g$ commutes with the group of complex scalars $\mathbb{C}^*$, hence $g$ is $\bbC$-linear, i.e., $g\in \GL_{d}(\mathbb{C})$. If $d=1$ then $g\in \GL_{d}(\mathbb{C})=S$, and we are done. We thus assume $d\geq 2$. Since the derived subgroup $\SU(d)<\U(d)$ is also characteristic, we get that $g$ normalizes it as well. However, $\SU(d)$ is a maximal subgroup of $\SL_{d}(\mathbb{C})$, hence it equals its own normalizer in $\SL_{d}(\mathbb{C})$. It follows that the normalizer of $\SU(d)$ in $\GL_{d}(\mathbb{C})$ is $\mathbb{C}^*\cdot \SU(d)=S$,
thus $g\in S$. This proves the claim.

We conclude that indeed, up to an inner automorphism, $\tau$ is either trivial or complex conjugation.
\end{proof}

\begin{prop} \label{prop:autP}
Consider $\conj{P}=P/CZ\cong \bbR^*\cdot \U(n-1) \ltimes \mathbb{C}^{n-1}$ and let $\tau:P\to \conj{P}$ be a continuous homomorphism with one-dimensional kernel. Then there exists an inner automorphism $i:P\to P$ such that
\[
\mbox{either} \quad \tau=\theta \circ i
\quad \mbox{or} \quad \tau=c\circ \theta\circ i,
\]
where $\theta : P \to \conj{P}$ is the obvious quotient map and $c:\conj{P}\to \conj{P}$ corresponds to complex conjugation on $\bbR^*\cdot \U(n-1) \ltimes \mathbb{C}^{n-1}$ as discussed in Lemma~\ref{lem:aut(P)}.

In particular, $\tau$ is surjective, $\ker(\tau) = \ker(\theta) = CZ$ and, up to precomposing $\tau$ by an inner automorphism of $P$, $\tau(P\cap D)=\theta(P\cap D)$.
\end{prop}

\begin{proof}
We first note that Lemma~\ref{lem:normalinP} implies that $Z \le \ker(\tau) \le CZ$, since $\dim(U)=2n-1\geq 3$ and $Z\cong \mathbb{R}$ is the identity component of $CZ$. By dimension considerations, $\tau(P)$ is Zariski dense in $\conj{P}$, as $\conj{P}$ is Zariski connected \cite[\S 23.1, Cor.~B]{Humphreys}. Thus $C \leq \ker(\tau)$, as $\conj{P}$ has trivial center, hence $\ker(\tau)=CZ$ and $\tau$ induces a continuous injection $\conj{\tau}:\conj{P}\to \conj{P}$ such that $\tau=\conj{\tau}\circ \theta$.

A continuous injective endomorphism of a Lie group with finitely many connected components is surjective, so $\conj{\tau}$ is necessarily an automorphism. By Lemma~\ref{lem:aut(P)} we get that either $\bar{\tau}=c\circ \inn(\bar{g})$ or $\bar{\tau}=\inn(\bar{g})$ for some $\bar{g}\in \conj{P}$, where $c$ denotes complex conjugation. Fixing $g\in P$ such that $\bar{g}=\theta(g)$ we get that either $\tau =c\circ \theta\circ \inn(g)$ or $\tau=\theta \circ \inn(g)$, proving the first part of the proposition.

It remains to show that $\tau(P\cap D)=\theta(P\cap D)$ up to an inner automorphism. This follows from Lemma~\ref{lem:DcapP}, since $\theta(P\cap D)=\theta(MA)$ corresponds to the subgroup $\mathbb{C}^*\cdot \SU(n-1)\subset \mathbb{C}^*\cdot \SU(n-1)\ltimes \mathbb{C}^{n-1}$, which is invariant under complex conjugation.
\end{proof}

\section{Compatibility and measuring incompatibility}\label{ssec:Compatible}

While much of this section applies more broadly, in what follows $G$ will always denote $\SU(n,1)$.
We fix its matrix presentation as discussed in \S\ref{sec:sun1} and freely use the notation introduced there.
In this section, we recall the definition of compatibility from \cite[\S3.4]{BFMS} and study compatibility of certain groups over local fields with $G$. In this paper, we will also need to measure the extent to which compatibility fails. This leads us to begin with the following sequence of definitions.

\begin{defn} \label{def:comp}
Let $P$ be a minimal parabolic subgroup of $G$ and $U$ its unipotent radical. An \emph{incompatibility datum} for $G$ is a tuple $(k,\bfH,\bfJ,U^\prime,\tau)$, where $k$ is a local field, $\bfH$ is a $k$-algebraic group, $\bfJ<\bH$ is a nontrivial $k$-subgroup, $U^\prime<U$ is a nontrivial proper subgroup, and $\tau:P \rightarrow N_{\bfH}(\bfJ)/\bfJ(k)$ is a continuous homomorphism such that the Zariski closure of $\tau(U^\prime)$ is not equal to the Zariski closure of $\tau(U)$. Here $N_{\bfH}(\bfJ)$ denotes the normalizer of $\bfJ$ in $\bfH$. We then have:
\begin{itemize}

\item
When $k$, $\bfH$, and $\bfJ$ are fixed, $\tau$ is an \emph{incompatible homomorphism} for $\bfJ$ if there exists $U^\prime < U$ as above so the corresponding tuple forms an incompatibility datum for $G$. If no such $U^\prime$ exists, then $\tau$ is called a \emph{compatible homomorphism}.

\item
When $k$ and $\bfH$ are fixed, $\bfJ$ is an \emph{incompatible subgroup} if there exists an incompatible homomorphism for $\bfJ$. Otherwise $\bfJ$ is a \emph{compatible subgroup}.

\item
When $k$ is fixed, $\bfH$ is an \emph{incompatible $k$-group} for $G$ if there is a nontrivial $k$-subgroup $\bfJ < \bfH$ that is an incompatible subgroup for the pair $(k, \bfH)$. Otherwise we call $\bfH$ a \emph{compatible $k$-group}.

\item
A local field $k$ is \emph{incompatible} with $G$ if there exists an incompatible $k$-group. Otherwise $k$ is called \emph{compatible}.
\end{itemize}
\end{defn}

We will see an example of an incompatibility datum in Proposition~\ref{prop:comppu}. First we prove some general results. To stay as notationally consistent with \cite{BaderFurman} as possible, if $\bfA$, $\bfB$ are $k$-algebraic groups we use the notation $\bfA/\bfB(k)$ to mean $(\bfA/\bfB)(k)$, that is, the $k$-points of $\bfA/\bfB$.

\begin{prop}\label{prop:compk}
Every nonarchimedean local field is compatible with $G$.
\end{prop}

\begin{proof}
If $k$ is a nonarchimedean local field, then for any $k$-group $\bfH$ and $k$-subgroup $\bfJ<\bfH$, $N_{\bfH}(\bfJ)/\bfJ(k)$ is totally disconnected. Therefore every continuous homomorphism $\tau:P \rightarrow N_{\bfH}(\bfJ)/\bfJ(k)$ must be trivial on the connected subgroup $U<P$. Compatibility clearly follows.
\end{proof}

In view of Proposition~\ref{prop:compk} we will be concerned in the rest of this subsection with archimedean local fields, i.e., $k=\mathbb{R}$ or $k=\mathbb{C}$. We then have the following general lemma.

\begin{lemma} \label{lem:Spar}
Let $(k,\bfH,\bfJ,U^\prime,\tau)$ be an incompatibility datum for $G$ where $k$ is $\bbR$ or $\mathbb{C}$. Suppose that $\bfS$ is the Zariski closure in $\bfH$ of the preimage of $\tau(P)$ under the map $N_{\bfH}(\bfJ)(k) \to N_{\bfH}(\bfJ)/\bfJ(k)$. Then $\bfS$ is not reductive. In particular, if $\bfH$ is reductive then $\bfS$ is contained in a proper parabolic subgroup of $\bfH$.
\end{lemma}

\begin{proof}
We first observe that if $\bfH$ is reductive and $\bfS$ is not reductive then $\bfS$ is contained in a proper parabolic subgroup of $\bfH$ by \cite[\S3]{BorelTits}. It then suffices to show that $\bfS$ is not reductive.

By extending scalars if necessary, we can assume that $k=\mathbb{C}$. We identify algebraic groups with their $\mathbb{C}$-points, writing $H$ for $\bfH(\mathbb{C})$ with similar notation for the other groups. We assume that $S$ is reductive and will prove that its Lie algebra has a noncentral nilpotent ideal, which is a contradiction.

Incompatibility of $\tau$ implies that $\tau(U)$ is nontrivial, since $\tau(U^\prime)$ must be a proper subgroup of $\tau(U)$. Note that $[P,U]=U$, thus $\tau(U)$ is a noncentral normal nilpotent subgroup of $\tau(P)$. We now consider $G$, $H$, and their subgroups as Lie groups, $\tau$ as a morphism of Lie groups, and denote their Lie algebras by the corresponding Gothic letters. Since $\mathfrak{s}$ is reductive, the ideal $\mathfrak{j}\trianglelefteq\mathfrak{s}$ has a direct complement $\mathfrak{i}\trianglelefteq\mathfrak{s}$, i.e., $\mathfrak{s}\cong \mathfrak{i}\oplus \mathfrak{j}$. The image of the composition
\[
\mathfrak{u} \to \mathfrak{s}/\mathfrak{j} \to \mathfrak{i} \to \mathfrak{s}
\]
defines a nontrivial noncentral nilpotent ideal in $\mathfrak{s}$, giving the desired contradiction and thus proving the lemma.
\end{proof}

To prove Theorem~\ref{thm:SRht}(1) we will be concerned exclusively with the case $k=\mathbb{R}$ and $\bfH(\bbR)=\PU(r,s)$ with $r+s=n+1$. For the proof of Theorem \ref{thm:Siu}, we will care about the case $r+s< n+1$. In particular, we will need the following technical result.

\begin{prop}\label{prop:comppu}
Let the tuple $(k,\bfH,\bfJ,U^\prime,\tau)$ be an incompatibility datum for $G$ with $k=\mathbb{R}$ and $\bfH=\PU(r,s)$ where $r+s \le n+1$, $r\geq s \ge 1$. Then
\begin{enumerate}
\item $r = n$ and $s=1$, i.e., $\bfH=\PU(n,1)$, and
\item $\bfJ$ is the center of the unipotent radical of a proper parabolic subgroup of $\bfH$ and $\tau$ has one-dimensional kernel.
\end{enumerate}
In particular, $\bfH=\PU(r,s)$ is compatible with $G$ for either $r+s < n+1$ or $s>1$.
\end{prop}

\begin{proof}
We work directly with real points of real algebraic groups and abandon the bold face notation. We again let $S$ denote the Zariski closure of the preimage of $\tau(P)$ under the map $N_H(J)\to N_H(J)/J$. Lemma~\ref{lem:Spar} implies that $S \le N_H(J)$ is contained in some maximal proper parabolic subgroup $Q < H$. Set $N=N_Q(J)$ and note that $S \le N \le Q$. Restricting the codomain, we consider $\tau$ as a map from $P$ to $N/J \le N_H(J) / J$.

The group $\U(r,s)$ acts on $\mathbb{C}^{r+s}$ preserving the hermitian form
\[
h_{r,s}(x_1,\dots,x_{n+1}) = \sum_{i=1}^r |x_i|^2 - \sum_{i=1}^s |x_{r+i}|^2,
\]
and $Q$ is the projectivization of the stabilizer in $\U(r,s)$ of a totally isotropic complex subspace $V \subset \mathbb{C}^{r+s}$. Let $d=\dim_\mathbb{C}(V)$ and note that $d\leq s$,
as $s$ is the dimension of a maximal $h_{r,s}$-isotropic subspace $\mathbb{C}^{r+s}$.
A parabolic subgroup of $H$ is conjugate to $Q$ if it is the stabilizer of a $d$-dimensional totally isotropic subspace. A parabolic subgroup is opposite to $Q$ if it is the stabilizer of a maximal totally isotropic subspace $V^\prime$ disjoint from $V$ such that the restriction of $h_{r,s}$ to $V \oplus V^\prime$ is nondegenerate.

A Levi factor of $Q$ is obtained by intersecting $Q$ with an opposite parabolic subgroup $Q^\prime$. If $V^\prime$ is the totally isotropic subspace associated with $Q^\prime$, then $L=Q\cap Q^\prime$ stabilizes the subspaces $V$, $V^\prime$ and $(V\oplus V^\prime)^\perp$, and these subspaces form a direct sum decomposition of $\mathbb{C}^{r+s}$. Taking $V\oplus V^\prime$ and $(V\oplus V^\prime)^\perp$ with the induced forms, we identify $\GL(V)$ with the stabilizer of $V$ in $\U(V\oplus V^\prime)$ under the action $T(v,v^\prime)=(Tv,(\conj{T}^{-1}){}^*v^\prime)$. This leads to the identification
\[
L\cong \mathrm{P}(\GL(V)\times \U((V\oplus V')^\perp)) \cong \mathrm{P}(\GL_d(\bbC)\times \U(r-d,s-d)),
\]
and we see that $K \cong \mathrm{P}(\U(d) \times \U(r-d) \times \U(s-d))$ is a maximal compact subgroup of $L$.

We use the notation introduced in \S\ref{subsec:subgroups} and recall that $M<P$ is locally isomorphic to $\U(n-1)$. The map $S\to S/J$ is real algebraic, thus the compact subgroup $\tau(M)<S/J$ has a compact lift $\wt{M}<S$ such that the map $S\to S/J$ restricts to a local isomorphism $\wt{M}\to \tau(M)$. Incompatibility of $\tau$ means that $U$ cannot be in the kernel of $\tau$, and it follows from Lemma~\ref{lem:normalinP} that $\ker(\tau) \le CZ$. In particular, $\tau$ is almost injective on $M$.

We therefore get that $\wt{M}$ is locally isomorphic to $\tau(M)$, which is locally isomorphic to $M$, therefore $\wt{M}$ is locally isomorphic to $\U(n-1)$. Let $L$ be a Levi factor of $Q$ that contains $\wt{M}$ and $K_0 < L$ be a maximal compact subgroup containing $\wt{M}$. We conclude that $d$, $s$, $r$ and $n$ are parameters satisfying $1\leq d \leq s \leq r \leq n$ and $r+s\leq n+1$ such that the group $\mathrm{P}(\U(d) \times \U(r-d) \times \U(s-d))$ contains a subgroup locally isomorphic to $\U(n-1)$.

We claim that $r = n$, which will imply $s=1$ and will prove (1). Assume for contradiction $r\leq n-1$. Note that the possibility $d=s=r$ is excluded. Indeed, in this case $\mathrm{P}(\U(d) \times \U(r-d) \times \U(s-d))=\PU(d)$ does not contain a subgroup locally isomorphic to $\U(n-1)$ by dimension considerations, as $d\leq r \leq n-1$. If $n=2$ we necessarily have $1= d = s = r$, which gives a contradiction. We thus have $n\geq 3$. Thus the commutator subgroup of the subgroup of $\mathrm{P}(\U(d) \times \U(r-d) \times \U(s-d))$ which is locally isomorphic to $\U(n-1)$ is almost simple, hence it projects almost injectively to one of the groups $\PU(d)$, $\PU(r-d)$, or $\PU(s-d)$. Since this commutator subgroup is locally isomorphic to $\SU(n-1)$, it follows by dimension considerations that $n-1$ is less than or equal to one of the three numbers $d$, $r-d$, or $s-d$. Since $s-d \leq r-d \leq n-2$ we must have that $n-1\leq d$ and we conclude that $d=s=r=n-1$. This gives a contradiction. Hence we must have $r=n$, completing the proof of (1).

%
%
%

We now have that $H=\PU(n,1)$. In particular, we have a natural surjection
\[
G=\SU(n,1) \to G/C \cong \PU(n,1)=H,
\]
with finite kernel. Let $P^\prime$ be the preimage of $Q$ under this surjection. Then $P'<G$ is a proper parabolic subgroup, and hence is conjugate to $P$. Precomposing the above map with the corresponding conjugation we get a surjection $\theta:G\to H$ such that $P=\theta^{-1}(Q)$ and $C=\theta^{-1}(e)$.

Therefore $P$ is locally isomorphic to $Q$ and $\dim(Q)=\dim(P)$. Note that $\tau$ factors through an injection $P/\ker(\tau) \to N/J = N_Q(J) / J$. Recalling that $\ker(\tau) \le CZ$, we have that $\dim(\ker(\tau))\leq 1$. We then have the following chain of inequalities
\begin{align}\label{dimN}
\dim(P)-1 &\leq \dim(P)-\dim(\ker(\tau)) = \dim(P/\ker(\tau)) \\
&\leq \dim(N)-\dim(J) = \dim(N/J) \nonumber \\
&\leq \dim(Q)-\dim(J) \nonumber \\
&= \dim(P)-\dim(J) \nonumber
\end{align}
that all follow easily from the above observations.

We claim that $\dim(N)=\dim(Q)$. If not, then $\dim(N) < \dim(Q)$ and we see from \eqref{dimN} that $\dim(P)-1 < \dim(P) - \dim(J)$ so $\dim(J)=0$. In other words, $J<Q$ is finite, and so $\theta^{-1}(J)<P$ is also finite, hence compact. Then $\theta^{-1}(J)$ commutes with $Z<P$ by Lemma~\ref{lem:Zcent}. Setting $\conj{Z}=\theta(Z)$, $J$ commutes with $\conj{Z} < Q$, and in particular $\conj{Z} \le N$.

Since $Z\lhd P$ is normal, $\conj{Z}\lhd Q$ is normal and hence $J\conj{Z}\lhd N$ is normal. Consider the natural map $\pi:N/J\to N/J\conj{Z}$. Since $Z$ is one-dimensional, $\conj{Z}$ is one-dimensional, and so $\ker(\pi) = J\conj{Z}/J$ is one-dimensional. Recall that $\ker(\tau) \le CZ$, so it is at most one-dimensional. We conclude that $\ker(\pi\circ\tau)$ is at most two-dimensional. However, $\dim(U)\geq 3$, thus $\ker(\pi\circ\tau) \le CZ$ by Lemma~\ref{lem:normalinP}, and in particular $\ker(\pi\circ\tau)$ is at most one-dimensional.

Considering the injective map $P/\ker(\pi\circ \tau) \to N/J\conj{Z}$, we obtain the contradictory chain of inequalities:
\begin{align*}
\dim(P)-1 &\leq \dim(P)-\dim(\ker(\pi\circ \tau))=\dim(P/\ker(\pi\circ \tau)) \\
&\leq \dim(N)-\dim(J\conj{Z}) = \dim(N/J\conj{Z}) \\
&= \dim(N)-1 \\
&< \dim(Q)-1 \\
&=\dim(P)-1
\end{align*}
Therefore $\dim(N) = \dim(Q)$ as claimed.

We now have that $N \le Q$ and $\dim(N)=\dim(Q)$. Since $Q$ is Zariski connected (again, see \cite[\S 23.1]{Humphreys}), we conclude that $N=Q$ and $J\lhd Q$ is normal. From Equation \eqref{dimN} we get that $\dim(J)\leq 1$. If $J$ was finite, it would be central in $Q$, since $J\lhd Q$ is normal, $Q$ is Zariski connected, and any discrete normal subgroup of a connected group is central. However $Q$ has trivial center and $J$ is nontrivial by the incompatibility assumption, thus we have that $\dim(J) = 1$.

Then $\theta^{-1}(J)$ is normal in $P$ and one-dimensional. Lemma~\ref{lem:normalinP} implies that $\theta^{-1}(J) \le CZ$ and we conclude that $J \le \conj{Z}=\theta(CZ)$. Since $\dim(J)=1$ and $\conj{Z}\cong Z\cong \mathbb{R}$ is connected, we must have $J=\conj{Z}$. Noting that $Q<H$ is a proper parabolic subgroup and $\conj{Z}<Q$ is the center of its unipotent radical, we have shown that $J$ is indeed the center of the unipotent radical of a proper parabolic subgroup of $H$. Finally, using the fact that $\dim(J)=1$ we get from Equation \eqref{dimN} that $\dim(\ker(\tau))=1$. This completes the proof of (2).
\end{proof}

\section{Preliminaries for the proof of Theorem~\ref{thm:SR}}\label{sec:map}

\subsection{Existence of equivariant maps}\label{ssec:FindMap}

Assume that $k$ is a local field and $\bfH$ is a connected adjoint $k$-algebraic group. Let $\Gamma<G$ be a lattice and consider a homomorphism $\rho : \Gam \to \bfH(k)$ with unbounded, Zariski dense image. Assume that $W<G$ is a closed noncompact subgroup for which there exists a faithful irreducible representation of $\bfH(k)$ on a $k$-vector space $V$ of dimension at least two and a $W$-invariant ergodic measure $\nu$ on the bundle $(G \times \bbP(V)) / \Gam$ that projects to Haar measure on $G / \Gam$, where the action of $W$ is induced by the left action of $W$ on $G$. The proof of the following result is exactly the same as the proof of \cite[Prop.\ 4.1]{BFMS}.

\begin{prop}\label{prop:map}
Under the assumptions above, there is a proper noncompact $k$-algebraic subgroup $\bfL < \bfH$ and a measurable $W$-invariant and $\Gam$-equivariant map $\phi : G \to \bfH / \bfL(k)$. We can also view $\phi$ as a measurable $\Gam$-map from $W \bs G$ to $\bfH / \bfL(k)$.
\end{prop}

\subsection{Algebraic representations}\label{ssec:BaderFurman}

Here we recall the work of Bader and Furman \cite{BaderFurman} that will be used in the proof of Theorem~\ref{thm:SR}. We also refer to \cite[\S4.2]{BFMS} for further discussion. We again assume that $k$ is a local field and $\bfH$ is a connected adjoint $k$-algebraic group.

We fix a lattice $\Gamma<G$ and a Zariski dense representation $\rho : \Gam \to \bfH(k)$. For a closed subgroup $T < G$, a \emph{$T$-algebraic representation of $G$} consists of:
\begin{itemize}

\item a $k$-algebraic group $\bfI$,

\item a $k$-$(\bfH \times \bfI)$-algebraic variety $\bfV$ that is a left $\bfH$-space and a right $\bfI$-space on which the $\bfI$-action is faithful,

\item a Zariski dense homomorphism $\tau : T \to \bfI(k)$,

\item an \emph{algebraic representation} of $G$ on $\bfV$, by which we mean an almost-everywhere defined measurable map $\phi : G \to \bfV(k)$ such that
\[
\phi(t g \gam^{-1}) = \rho(\gam) \phi(g) \tau(t)^{-1}
\]
for every $\gam \in \Gam$, every $t \in T$, and almost every $g \in G$.
\end{itemize}
The data for a $T$-algebraic representation is denoted by $\bfI_\bfV$, $\tau_\bfV$, and $\phi_\bfV$.

A $T$-algebraic representation is \emph{coset $T$-algebraic representation} when $\bfV$ is the coset space $\bfH / \bfJ$ for some $k$-algebraic subgroup $\bfJ < \bfH$ and $\bfI$ is a $k$-subgroup of $N_\bfH(\bfJ) / \bfJ$, where $N_\bfH(\bfJ)$ is the normalizer of $\bfJ$ in $\bfH$ and $\bfH / \bfJ$ is endowed with the standard right action of $N_\bfH(\bfJ) / \bfJ$. The collection of $T$-algebraic representations of $G$ forms a category. The Howe--Moore theorem implies that if $T$ is noncompact then the $T$-action on $G / \Gam$ is mixing, hence weakly mixing. In this case, by \cite[Thm.\ 4.3]{BaderFurman}, the above category has an initial object which is a coset $T$-algebraic representation.

\begin{defn}\label{def:gate}
Suppose that $G$, $\Gam$, and $T$ are as above. An initial object in the category of $T$-algebraic representations of $G$ is called
the \emph{gate} or, when wishing to emphasize its dependence on $T$, the \emph{$T$-gate}.
\end{defn}
\begin{rem}
As remarked in \cite[Def.\ 4.4]{BaderFurman2} there is not a unique $T$-gate in the category of $T$-algebraic representations. Indeed any conjugate of an initial coset $T$-algebraic representation will yield another $T$-gate (see \cite[Rem. 4.5]{BaderFurman}). The gate is easily seen to be unique up to unique isomorphism and for this reason, and to remain consistent with the discussion in \cite{BaderFurman2}, we continue to refer to such a choice as \emph{the} $T$-gate.
\end{rem}

We recall one last definition from \cite[\S 4.2]{BFMS}. Given closed noncompact subgroups $S, T < G$, consider their gates $\phi_S, \bfV_S$ and $\phi_T, \bfV_T$. We say that these gates \emph{have the same map} if $\bfV_S=\bfV_T$ and $\phi_T, \phi_S$ agree almost everywhere. We then have the following.

\begin{lemma}[cf.~{\cite[Lem.\ 4.4]{BFMS}}] \label{lem:BFMS4.4}
Assume that $T<G$ is a closed noncompact subgroup. Assume $T=T_0\lhd T_1\lhd \cdots \lhd T_n=T^\prime$ is a sequence of subgroups of $G$ such that $T_{i-1}$ is normal in $T_{i}$ for each $i=1,\ldots,n$. Then the gates for $T$ and $T^\prime$ can be chosen to have the same map.
\end{lemma}

\begin{proof}
By induction we can assume that $n=1$, that is $T< T^\prime \le N_G(T)$. Then \cite[Lem.\ 4.4]{BFMS} implies we can assume that the gates for $T$ and $N_G(T)$ have the same map $\phi:G\to \bfH/\bfJ(k)$. Since $\phi$ is an $N_G(T)$-algebraic representation, it is also a $T^\prime$-algebraic representation. Let $\phi^\prime:G\to \bfH/\bfJ^\prime(k)$ be the gate in the category of $T^\prime$-algebraic representations. It follows that we can find an $\bfH$-equivariant $k$-map $\bfH/\bfJ^\prime \to \bfH/\bfJ$. Viewing $\phi^\prime$ as a $T$-algebraic representation, we can also find an $\bfH$-equivariant $k$-map $\bfH/\bfJ \to \bfH/\bfJ^\prime$. By \cite[Cor.\ 4.7]{BaderFurman} we see that the gates for $T$ and $T^\prime$ can indeed be chosen to have the same map.
\end{proof}

\section{The proof of Theorem~\ref{thm:SR}\eqref{thm:SRcomp}}\label{sec:SRcomp}

In this section we state and prove Proposition~\ref{prop:SRincompt} and then use it to prove Theorem~\ref{thm:SR}\eqref{thm:SRcomp}. Proposition~\ref{prop:SRincompt} will also be used in the proof of Theorem~\ref{thm:SR}\eqref{thm:SRpun1} in the next section. Throughout this section we rely on the notation, definitions, and results of the previous sections, particularly the notion of incompatibility given in Definition~\ref{def:comp} and the notion of a gate from Definition~\ref{def:gate}.

\begin{prop} \label{prop:SRincompt}
Let $G$ be $\SU(n,1)$ for some $n \ge 2$, $\Gam < G$ be a lattice, and $W < G$ be a noncompact connected almost simple subgroup. Suppose that $k$ is a local field, $\bfH$ is a connected adjoint $k$-algebraic group, and that $\rho : \Gam \to \bfH(k)$ is a homomorphism with unbounded, Zariski dense image. Assume moreover that there is a faithful irreducible representation of $\bfH(k)$ on a $k$-vector space $V$ of dimension at least two and a $W$-invariant, ergodic measure $\nu$ on $(G \times \bbP(V)) / \Gam$ that projects to Haar measure on $G / \Gam$.

Suppose that $U^\prime<W$ is a nontrivial unipotent subgroup, consider the category of $U^\prime$-algebraic representations of $G$, and let $\Psi : G \to \bfH / \bfJ(k)$ be the corresponding gate, where $\bfJ < \bfH$ is a $k$-algebraic subgroup and $\Psi$ is a measurable map that is $(U^\prime \times \Gam)$-equivariant with respect to a homomorphism $\tau : U^\prime \to N_\bfH(\bfJ) / \bfJ(k)$. Then $\tau$ extends to a continuous homomorphism $\conj{\tau}:P \to N_\bfH(\bfJ) / \bfJ(k)$ with $\conj{\tau}|_{U^\prime}=\tau$ such that $\Psi$ is $(P \times \Gam)$-equivariant with respect to $\conj{\tau}$. Furthermore:
\begin{enumerate}

\item If $\bfJ$ is trivial, then $\rho$ extends to a continuous homomorphism from $G$ to $\bfH(k)$.

\item If $\bfJ$ is nontrivial, then it is an incompatible subgroup of $\bfH$ and $\conj{\tau}$ is an incompatible homomorphism.

\end{enumerate}
\end{prop}

We note that Proposition~\ref{prop:SRincompt} is essentially proved in \cite[\S4.3]{BFMS}. The main difference is that here we emphasize (in)compatibility of $\bfJ$ and $\conj{\tau}$, rather than compatibility of the group $\bfH$. Apart from that and minor differences due to the fact that $G=\SU(n,1)$ here rather than $\SO(n,1)$, the proofs are the same. Nevertheless, we include the full proof for the reader's convenience.

\begin{proof}[Proof of Proposition~\ref{prop:SRincompt}]
By \cite[\S3]{BorelTits} we can find a proper parabolic subgroup $P<G$ containing $U^\prime$. Let $M$, $A$, $U$, and $Z$ be as in \S\ref{subsec:subgroups}. Since $G$ has real rank 1, $P$ is minimal parabolic and all its unipotent elements are contained in $U$, hence $U^\prime<U$.

Now consider the group $U^\prime Z<U$, where $Z$ is the center of $U$. We then have $U^\prime \lhd U^\prime Z \lhd U \lhd P$, and applying Lemma~\ref{lem:BFMS4.4} we see that the gates for $U^\prime$ and $P$ can be chosen to have the same map. Therefore, $\tau$ extends to a continuous homomorphism $\conj{\tau}:P \to N_\bfH(\bfJ) / \bfJ(k)$ with respect to which $\Psi$ is $(P \times \Gam)$-equivariant.

Assume $\bfJ$ is trivial. As $A < P$, $\Psi$ is $A$-equivariant through $\conj{\tau}|_A$. Triviality of $\bfJ$ implies that this must be the gate in the category of $A$-algebraic representations. Then Lemma~\ref{lem:BFMS4.4} implies that $\conj{\tau}|_A$ extends to a homomorphism $N_G(A) \rightarrow \bH(k)$ for which $\Psi$ is $N_G(A)$-equivariant, where $N_G(A)$ is the normalizer of $A$ in $G$.

Since $N_G(A)$ contains a Weyl element for $A$ (i.e., a generator for the Weyl group) we have that the group $\langle P, N_G(A) \rangle$ generated by $P$ and $N_G(A)$ is $G$. Since $\Psi$ is equivariant for both $P$ and $N_G(A)$, using \cite[Prop.\ 5.1]{BaderFurman} and following the end of the proof of \cite[Thm.\ 1.3]{BaderFurman}, we conclude that $\rho : \Gam \to \bfH(k)$ extends to a continuous homomorphism $\wh{\rho} : G \to \bfH(k)$. This proves the proposition when $\bfJ$ is trivial.

We now assume that $\bfJ$ is nontrivial but $\conj{\tau}$ is compatible and derive a contradiction. Compatibility of $\conj{\tau}$ implies that the Zariski closures of $\conj{\tau}(U^\prime)$ and $\conj{\tau}(U)$ coincide. Proposition~\ref{prop:map} implies that there exists a proper noncompact $k$-algebraic subgroup $\bfL\lneq \bfH$ and a measurable $W$-invariant, $\Gam$-equivariant map $\phi : G \to \bfH / \bfL(k)$.

The $W$-invariant map $\phi$ is also $U^\prime$-invariant, as $U^\prime<W$. Thus it factors through $\Psi:G\to \bH/\bfJ(k)$ and
\[
G\to \bH/\bfJ(k) \to (\bH/\bfJ)/\overline{\bar{\tau}(U')}(k)= (\bH/\bfJ)/\overline{\bar{\tau}(U)}(k)
\]
by $U^\prime$-invariance, where $\overline{\conj{\tau}(U^\prime)} = \overline{\conj{\tau}(U)}$ are the corresponding Zariski closures. Since $\Psi$ is $U$-equivariant, the above composition is $U$-invariant, and it follows that $\phi$ is also $U$-invariant. Since $\phi$ is also $W$-invariant and $\langle U, W \rangle = G$, we obtain that $\phi : G \to \bH / \bfL(k)$ is an essentially constant $\Gam$-equivariant map, hence $\rho(\Gam)$ has a fixed point on $\bH / \bfL(k)$. This is impossible since $\rho(\Gamma)$ is Zariski dense in $\bH$ and $\bfL$ is a proper algebraic subgroup of the connected adjoint group $\bH$. This is a contradiction, which completes the proof.
\end{proof}

\noindent
We are now prepared to prove part \eqref{thm:SRcomp} of Theorem~\ref{thm:SR}.

\begin{proof}[Proof of Theorem~\ref{thm:SR}\eqref{thm:SRcomp}]
Suppose $U^\prime<W$ is a nontrivial unipotent subgroup and consider the category of $U^\prime$-algebraic representations with gate $\Psi : G \to \bfH / \bfJ(k)$, where $\bfJ < \bfH$ is a $k$-algebraic subgroup and $\Psi$ is a measurable map that is $(U^\prime \times \Gam)$-equivariant with respect to a homomorphism $\tau : U' \to N_\bfH(\bfJ) / \bfJ(k)$.

We claim that $\bfJ$ is trivial. Indeed, if it was nontrivial then it would be incompatible by Proposition~\ref{prop:SRincompt}, contradicting the compatibility of $\bfH$ with $\SU(n,1)$. Applying Proposition~\ref{prop:SRincompt} again, we see that $\rho$ extends to a continuous homomorphism from $G$ to $\bfH(k)$. This completes the proof.
\end{proof}

\section{The proof of Theorem~\ref{thm:SR}\eqref{thm:SRpun1}}\label{sec:SRpun1}

We begin by stating Proposition~\ref{prop:SRJZ} and then use it and Proposition~\ref{prop:SRincompt} to prove Theorem~\ref{thm:SR}\eqref{thm:SRpun1}. The rest of this section is then devoted to the proof of Proposition~\ref{prop:SRJZ}.

\begin{prop} \label{prop:SRJZ}
Suppose that $G$ is $\SU(n,1)$ for $n \ge 2$ and $\Gam < G$ is a lattice. Let $H=\PU(n,1)$, $\conj{Z}<H$ be the center of the unipotent radical of a proper parabolic subgroup $Q<H$, and $\rho:\Gamma\to H$ be a homomorphism with unbounded, Zariski dense image. Assume that there exists a continuous homomorphism $\tau:P\to Q/\conj{Z}$ with one-dimensional kernel and a measurable map $\Phi:G\to H/\conj{Z}$ that is $(P \times \Gam)$-equivariant with respect to the left $\Gamma$-action and the right $P$-action on $H / \conj{Z}$ via $\tau$. Then $\rho$ extends to a continuous homomorphism from $G$ to $H$.
\end{prop}

Using this, we prove part (2) of Theorem~\ref{thm:SR}.

\begin{proof}[Proof of Theorem~\ref{thm:SR}\eqref{thm:SRpun1} given Proposition~\ref{prop:SRJZ}]
Here $k=\mathbb{R}$ and $\bfH$ is the $k$-algebraic group corresponding to $H$. Let $U^\prime<W$ be a nontrivial unipotent subgroup and consider the category of $U'$-algebraic representations of $G$ and the corresponding gate $\Psi : G \to \bfH / \bfJ(k)$, where $\bfJ < \bfH$ is a $k$-algebraic subgroup and $\Psi$ is a measurable map that is $(U^\prime \times \Gam)$-equivariant  with respect to a homomorphism $\tau : U^\prime \to N_\bfH(\bfJ) / \bfJ(k)$. By Proposition~\ref{prop:SRincompt}, $\tau$ extends to a continuous homomorphism $\conj{\tau}:P \to N_\bfH(\bfJ) / \bfJ(k)$ such that $\Psi$ is $(P \times \Gam)$-equivariant with respect to $\bar{\tau}$.

If $\bfJ$ is trivial then we are done by Proposition~\ref{prop:SRincompt}. We therefore assume that $\bfJ$ is nontrivial. It follows, again by Proposition~\ref{prop:SRincompt}, that $\conj{\tau}$ is incompatible. Proposition~\ref{prop:comppu} implies that $\bfJ$ is the center of the unipotent radical of a proper parabolic subgroup of $\bfH$ and $\ker(\conj{\tau})$ is one-dimensional. By Proposition~\ref{prop:SRJZ} we conclude that $\rho$ indeed extends to a continuous homomorphism from $G$ to $\bfH(k) = \PU(n,1)$. This completes the proof.
\end{proof}

The proof of Proposition~\ref{prop:SRJZ} will be given in \S\ref{ssec:proofSRJS}. Before giving the proof, we require preliminary subsections discussing the notion of \emph{fiber products} in the measured category and the incidence geometry of \emph{chains} on the ideal boundary of complex hyperbolic space.

\subsection{Fiber products}\label{ssec:Fiber}

We begin by recalling some standard definitions. We will consider the category of Lebesgue spaces and their morphisms, where a Lebesgue space is a standard Borel space endowed with a measure class. An almost everywhere defined measurable map between Lebesgue spaces is said to be measure class preserving if the preimage of every null set is null. A morphism of Lebesgue spaces is an equivalence class of almost everywhere defined measure class preserving measurable maps, where two such maps are equivalent if they agree almost everywhere.

We now recall the definition of the fiber product in this category. Suppose that $X, Y, Z$ are Lebesgue spaces endowed with probability measures $\mu_X$, $\mu_Y$, and $\mu_Z$, respectively. Let $\phi : X \to Z$ and $\psi : Y \to Z$ be maps such that the measures are compatible in the sense that $\phi_* \mu_X$ and $\psi_*\mu_Y$ are in the same measure class as $\mu_Z$. The corresponding set theoretic fiber product is
\[
X \times_Z Y = \left\{(x, y) \in X \times Y\ :\ \phi(x) = \psi(y) \right\} \subseteq X \times Y.
\]
For $\mu_Z$ almost every $z\in Z$, disintegration of $\phi$ and $\psi$ give  measures $\nu_{X,z}$ and $\nu_{Y,z}$ on the corresponding fibers $\phi^{-1}(z)\subset X$ and $\psi^{-1}(z)\subset Y$ such that
\[
\mu_X = \int_Z \nu_{X,z}\, d\mu_Z \quad\quad\quad \mu_Y = \int_Z \nu_{Y,z}\, d\mu_Z.
\]
We then define a measure
\[
\mu_X \times_Z \mu_Y = \int_Z (\nu_{X,z} \times \nu_{Y,z})\, d\mu_Z,
\]
on $X\times Y$ whose equivalence class is supported on the set theoretic fibered product $X \times_Z Y$. Further, note that this measure class is independent of the choices of representatives for $\phi$ and $\psi$ and the choices of representatives for $\mu_X$, $\mu_Y$, and $\mu_Z$. We thus view this measure as a measure class on $X \times_Z Y$ and call it the \emph{fiber product measure}. For details, see for instance \cite[p.\ 265]{Ramsay}. We will need the following lemma.

\begin{lemma} \label{lem:fiberedproduct}
Consider Lebesgue spaces $X$, $Y$, and $Z$ endowed with morphisms $X\to Z$ and $Y\to Z$. Assume $X^\prime$, $Y^\prime$, and $Z^\prime$ are standard Borel spaces endowed with Borel maps $X^\prime\to Z^\prime$ and $Y^\prime\to Z^\prime$, and moreover that there are almost everywhere defined measurable maps $f:X \to X^\prime$, $g:Y\to Y^\prime$, and $h:Z\to Z^\prime$ such that the following diagram commutes:
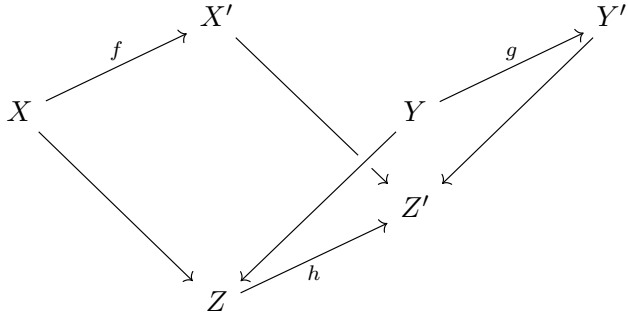
\[
	\begin{tikzcd}
{} & & X^\prime \arrow{ddrr} & & & & Y^\prime  \arrow{ddll} \\
X \arrow[crossing over]{ddrr} \arrow{urr}[above]{f} & & & & Y \arrow[crossing over]{ddll} \arrow{urr}[above]{g} & & \\
& & & & Z^\prime & & \\
& & Z \arrow{urr}[below]{h} & & & &
	\end{tikzcd}
\]
Consider the product map $f \times g: X \times Y \to X^\prime \times Y^\prime$ and the fiber product measure class $\mu_X \times_Z \mu_Y$ on $X\times_Z Y \subset X\times Y$. Then for $\mu_X \times_Z \mu_Y$ almost every $(x,y)$, one has that $(f \times g)(x,y) \in X^\prime \times_{Z^\prime} Y^\prime$.
\end{lemma}

\begin{proof}
Using the choices and notation introduced above, we have
\[
(f \times g)_*(\mu_X \times_Z \mu_Y) = \int_{Z^\prime} (f_* \nu_{X,z} \times g_* \nu_{Y,z})\, d h_*\mu_Z,
\]
and conclude that $(f \times g)_*(\mu_X \times_Z \mu_Y)=f_*\mu_X \times_{Z^\prime} g_*\mu_Y$. That is, the push forward of the fiber product measure is the fiber product of the push forward measures. It follows that this measure is indeed supported on the set theoretical fiber product $X^\prime \times_{Z^\prime} Y^\prime$.
We then get that:
\begin{align*}
&\,(\mu_X \times_Z \mu_Y)\left((f \times g)^{-1}((X^\prime \times Y^\prime) \ssm (X^\prime\times_{Z^\prime} Y^\prime))\right) \\
= &\,(f \times g)_*(\mu_X \times_Z \mu_Y)((X^\prime \times Y^\prime) \ssm (X^\prime\times_{Z^\prime} Y^\prime))\\
= &\,(f_*\mu_X \times_{Z^\prime} g_*\mu_Y)((X^\prime \times Y^\prime) \ssm (X^\prime\times_{Z^\prime} Y^\prime)) \\
=&\,0
\end{align*}
This proves the lemma.
\end{proof}

\subsection{Real algebraic varieties as Lebesgue spaces} \label{ssec:realvarieties}

Assume $V$ is a real algebraic variety. Any choice of smooth volume form on the nonsingular locus of $V$ gives rise to a same measure class on $V$ which we will call the \emph{volume measure class} and denote $\mu_V$. Thus $V$ has a natural structure as a Lebesgue space. We will use the following fact.

\begin{lemma} \label{lem:fpalv}
Let $X$, $Y$, and $Z$ be real algebraic varieties and let $\phi:X\to Z$, $\psi:Y\to Z$ be real rational maps that are surjective on real points such that $\phi_* \mu_X$ and $\psi_*\mu_Y$ are in the same measure class as $\mu_Z$. Consider the corresponding fiber product $X\times_Z Y$ in the category of real algebraic varieties. Then the volume measure class on $X\times_Z Y$ coincides with the corresponding fiber product measure class of the volume measure classes on $X$, $Y$, and $Z$.
\end{lemma}

\subsection{The incidence geometry of chains}\label{ssec:Chains}

Let $\bbB^n$ denote complex hyperbolic $n$-space, and recall the notation from \S\ref{subsec:subgroups}. Then $G = \SU(n,1)$ acts transitively on $\bbB^n$ with point stabilizers conjugates of $K = \mathrm{S}(\U(n) \times \U(1))$, the maximal compact subgroup of $G$. We can also identify $\bbB^n$ with the unit ball in $\bbC^n$ with its Bergman metric. Note that $\bbB^1$ is naturally identified with the Poincar\'e disk model of hyperbolic $2$-space.

Let $\partial\bbB^n$ denote the ideal boundary of $\bbB^n$, which is homeomorphic to the $(2n-1)$-sphere $S^{2n-1}$. Identify $\partial\bbB^n$ with the space of isotropic lines in $\mathbb{C}^{n+1}$ with respect to a hermitian form $h$ of signature $(n,1)$, thus with the real algebraic variety $P \bs G$ for $P < G$ a minimal parabolic subgroup. The space of pairs of points $\partial\bbB^n\times \partial\bbB^n$ is then identified with $P \bs G \times P \bs G$ and the Zariski open subset $(\partial\bbB^n)^{(2)}$ consisting of pairs of distinct points is identified with $MA \bs G$, i.e., the space of oriented geodesics in $\bbB^n$.

A totally geodesic holomorphic embedding $f : \bbB^1 \hookrightarrow \bbB^n$ induces an embedding
\[
f_\infty : \partial \bbB^1 \simeq S^1 \hookrightarrow \partial \bbB^n,
\]
and $f_\infty(\partial \bbB^1)$ is called a \emph{chain} on $\partial \bbB^n$. Denote the space of chains on $\partial \bbB^n$ by $\mathcal{C}$. Chains were originally studied by Cartan \cite{Cartan} and we refer to \cite{Goldman} for basic facts about them. Chains are in one-to-one correspondence with two-dimensional subspaces of $\mathbb{C}^{n+1}$ on which $h$ is nondegenerate of signature $(1,1)$. In particular, the $G$-action on $\mathcal{C}$ is transitive and the stabilizers are conjugate in $G$ to $D=\mathrm{S}(\U(1,1) \times \U(n-1))$. Therefore, we have the identification $\mathcal{C} \simeq D \bs G$.

Given a natural number $k$, we denote by $\mathcal{C}_k$ the space whose elements are a chain and $k$ points lying on that chain. That is,
\[
\mathcal{C}_k = \left\{(c,x_1,\ldots,x_k)\ :\ c\in \mathcal{C},\ x_1,\ldots,x_k\in c\right\}.
\]
Let $\mathcal{C}_{(k)}$ denote the subset of $\mathcal{C}_k$ where the points are distinct, i.e.,
\[
\mathcal{C}_{(k)}= \left\{(c,x_1,\ldots,x_k) \in \mathcal{C}_k\ :\ x_i\neq x_j \text{ for } i\neq j \right\}.
\]
Note that $\mathcal{C}_0=\mathcal{C}$ and $\mathcal{C}_1=\mathcal{C}_{(1)}$ is the set of pairs consisting of a chain and a point on that chain. Thus $\mathcal{C}_0\simeq D \bs G$ and $\mathcal{C}_1 \simeq (P\cap D) \bs G = MAZ \bs G$.
In particular, these spaces have real algebraic variety structures and we endow them with the corresponding volume measure classes, which are also the unique $G$-invariant measure classes. We have the projection $\mathcal{C}_1\to \mathcal{C}$ forgetting the point, which corresponds to the projection $(P\cap D)\bs G \to D \bs G$. This map is real algebraic and $G$-equivariant. The space $\mathcal{C}_k$ can be seen as the $k^{th}$ fiber power of the above projection and is thus also endowed with a natural real algebraic variety structure and a corresponding measure class. The space $\mathcal{C}_{(k)}$ is a dense Zariski open subset of $\mathcal{C}_k$ and then is also endowed with a real algebraic variety structure and a corresponding measure class. In particular, the inclusion $\mathcal{C}_{(k)}\subset \mathcal{C}_k$ is an isomorphism of Lebesgue spaces.

In this paper we will primarily be concerned with $\mathcal{C}_{(2)}$ and $\mathcal{C}_{(3)}$. A distinct pair of points $x,y \in \partial\bbB^n$ determines a unique chain, which is easily seen by considering $x$ and $y$ as isotropic lines in $\mathbb{C}^{n+1}$ and taking their span. Thus there are $G$-equivariant isomorphisms $MA \bs G \simeq (\partial\bbB^n)^{(2)} \simeq \mathcal{C}_{(2)}$.

We now consider the subset
\[
\mathcal{C}^0_k=\left\{(c,x_1,\ldots,x_k) \in \mathcal{C}_k\ :\ x_i \neq x_j \textrm{ for some } i \neq j \right\} \subset \mathcal{C}_k
\]
consisting of $k$-tuples of points containing at least two distinct points and the subset
\[
\mathcal{C}^1_k = \left\{(c,x_1,\ldots,x_k) \in \mathcal{C}_k\ :\ x_i\neq x_1 \textrm{ for } i \neq 1\right\} \subset \mathcal{C}_k^0
\]
where the first point in each chain is different from the others. We note that
\[
\mathcal{C}_{(k)} \subset \mathcal{C}^1_k \subset \mathcal{C}^0_k \subset \mathcal{C}_k,
\]
and these are all dense Zariski open subsets of $\mathcal{C}_k$. In particular, they are all isomorphic in the category of Lebesgue spaces.

\begin{rem}
The space $\mathcal{C}^0_k$ can be alternatively described as:
\[
\left\{(x_1,\ldots,x_k)\ :\  x_i \subset \bbC^{n+1}\ \textrm{an isotropic line},\ \dim_\bbC \mathrm{span} \{ x_1,\ldots,x_n \} = 2\right\}.
\]
It is this space that is considered in Pozzetti's chain rigidity theorem, which is Theorem~\ref{thm:Pozzetti} below.
\end{rem}

We now make an observation that will be useful later. For each $k$ we define $\alpha_k:\mathcal{C}^1_k\to \mathcal{C}_1$ by $(c,x_1,\ldots,x_k)\mapsto (c,x_1)$. In particular we have the map $\alpha_2:\mathcal{C}_{(2)}=\mathcal{C}^1_2\to \mathcal{C}_1$ that forgets the second point of each chain. We observe that $\alpha_k:\mathcal{C}^1_k\to \mathcal{C}_1$ can be identified in the category of algebraic varieties with the $(k-1)^{st}$ fibered power of $\al_2$. It follows by Lemma~\ref{lem:fpalv} that $\mathcal{C}^1_k$ is isomorphic to the $(k-1)^{st}$ fibered power of $\al_2$ also in the category of Lebesgue spaces.

Note that $\mathcal{C}_{(2)}\simeq MA \bs G$, but $G$ is not transitive on $\mathcal{C}_{(k)}$ already for $k=3$. Indeed, $\mathcal{C}_{(3)}$ has two $G$-orbits distinguished by the cyclic orientations of the given triple of points on each chain. The stabilizer in $G$ of each element of $\mathcal{C}_{(3)}$ is a compact subgroup, namely a conjugate of $M<G$. It follows that the measure class on $\mathcal{C}_{(3)}$ is the sum of the unique $G$-invariant measure classes on the two $G$-orbits. Each of $\mathcal{C}^0_3$, $\mathcal{C}^1_3$, and $\mathcal{C}_3$ are obtained by adding some lower-dimensional manifolds, which are null sets, to $\mathcal{C}_{(3)}$.

\subsection{Pozzetti's chain rigidity theorem}

In \cite{Cartan}, Cartan studied the incidence geometry of chains in $\partial\bbB^n$ and showed that a map preserving this geometry must come from an isometry of $\bbB^n$. His work was later generalized by Burger--Iozzi to almost everywhere defined measurable maps under the added assumption that the map preserves orientation \cite{BurgerIozzi}. This assumption was later removed by Pozzetti \cite{Pozzetti}, and this is the result we will need.

In the formulation of the theorem below we use the notation introduced in \S\ref{ssec:Chains}. Recall that $\partial \bbB^n$ represents the variety of isotropic lines in $\mathbb{C}^{n+1}$, and under this identification, $\mathcal{C}^0_3$ represents triples of isotropic lines that span a two-dimensional subspace. Briefly, Pozzetti's chain rigidity theorem is that a measurable map from $\partial \bbB^n$ to $\partial \bbB^n$ that takes triples of points on a chain to triples of points on a chain is necessarily rational; for more on rationality in this general context see \cite[Ch.~3]{Zimmer}. Restated precisely in our language, this is the following.

\begin{thm}[Thm.\ 1.6 \cite{Pozzetti}]\label{thm:Pozzetti}
For $n \ge 2$, let $\phi : \partial \bbB^n \to \partial \bbB^n$ be a measurable map whose essential image is Zariski dense. Endow $(\partial \bbB^n)^3$ with the measure class $[\mu]$ associated with the volume form on $\mathcal{C}^0_3$ and consider the induced map $\phi^3:(\partial \bbB^n)^3 \to (\partial \bbB^n)^3$. If the $[\mu]$-essential image of $\phi^3$ is contained in $\mathcal{C}^0_3$, then $\phi$ agrees almost everywhere with a rational map.
\end{thm}

\begin{proof}[On the proof of Theorem~\ref{thm:Pozzetti}]
The statement of \cite[Thm.\ 1.6]{Pozzetti} assumes that the target of $\phi$ is the Shilov boundary $\calS_{m_1,m_2}$ associated with $\SU(m_1, m_2)$ when $1 < m_1 < m_2$. However, nowhere in the proof of the theorem does Pozzetti use the assumption that $1 < m_1$, and hence her result also holds for $\calS_{1, m}$, $m>1$, and in particular for $\calS_{1, n}$. The Shilov boundary is the space of maximal isotropic subspaces for the relevant hermitian form, hence $\calS_{1, n} \simeq \partial \bbB^n$. In other words, Pozzetti's proof applies without alteration to give Theorem~\ref{thm:Pozzetti} as stated. The reason for Pozzetti's assumption that $m_1 > 1$ is that the analogue of her main result was previously proved by Burger--Iozzi \cite{BurgerIozzi}, however Burger and Iozzi did not need a statement as strong as Theorem~\ref{thm:Pozzetti} to prove their main result.
\end{proof}

We recall from \S\ref{ssec:Chains} that $\mathcal{C}^1_3$ can be identified with the fibered square of the forgetful map $\al_2:\mathcal{C}_{(2)}\to \mathcal{C}_1$ and that $\mathcal{C}_{(2)}\simeq (\partial \bbB^n)^{(2)}$. Upon making this identification, we consider the map $\al:(\partial \bbB^n)^{(2)}\to \mathcal{C}_1$ and regard $\mathcal{C}^1_3$ as the fibered square of $\al$. We then obtain the following corollary of Theorem~\ref{thm:Pozzetti}.

\begin{cor} \label{cor:pozzetti}
Let $G$ be $\SU(n,1)$ for $n \ge 2$, $\Gam < G$ be a lattice, $H$ be $\PU(n,1)$, and $\rho:\Gamma\to H$ be a homomorphism with unbounded, Zariski dense image. Assume that $\phi : \partial \bbB^n \to \partial \bbB^n$ and $\psi:\mathcal{C}_1 \to \mathcal{C}_1$ are measurable, $\Gamma$-equivariant maps, where $\Gamma$ acts on the domain through its inclusion into $G$ and on the target via $\rho$.
Then the following assertions hold.
\begin{enumerate}
\item
The essential image of $\phi^2:(\partial \bbB^n)^{2}\to (\partial \bbB^n)^{2}$ is contained in the set $(\partial \bbB^n)^{(2)}$ of distinct points.

\item
Considering the restricted map $\phi^{(2)}:(\partial \bbB^n)^{(2)} \to (\partial \bbB^n)^{(2)}$, if $\al \circ \phi^{(2)}$ and $\psi\circ \al$ agree almost everywhere as maps from $(\partial \bbB^n)^{(2)}$ to $\mathcal{C}_1$, then $\rho$ extends to a continuous homomorphism from $G$ to $H$.

\end{enumerate}
\end{cor}

\noindent The content of $(2)$ in Corollary \ref{cor:pozzetti} is that the boundary map $\phi$ sends chains to chains and that the induced map on chains is $\psi$.

\begin{proof}
We first prove (1). Note that $(\partial \bbB^n)^{(2)}\subset (\partial \bbB^n)^{2}$ is open and Zariski dense, hence of full measure. Moreover $(\partial \bbB^n)^{(2)}$ is isomorphic to $MA \bs G$ as a $G$-space and so it is $\Gamma$-ergodic by Howe--Moore. It follows that the essential image of $\phi^2$ is either contained in $(\partial \bbB^n)^{(2)}$ or in its complement, which is the diagonal in $(\partial \bbB^n)^{2}$. If the latter were true then $\phi$ would be essentially constant, and its essential image is invariant under $\Gamma$. This would imply that $\rho(\Gamma)$ is contained in a proper parabolic subgroup of $H$, contradicting Zariski density of $\rho(\Gam)$. It therefore follows that the essential image of $\phi^2$ is contained in the set of pairs of distinct points, $(\partial \bbB^n)^{(2)}$. This proves (1).

To prove (2), assume that $\al\circ \phi^{(2)}$ and $\psi\circ \al$ agree almost everywhere. By a well-known lemma of Margulis \cite[Lem.\ 5.1.3]{Zimmer}, to show that $\rho$ extends it suffices to show that $\phi$ is rational. Therefore it suffices to show that Theorem~\ref{thm:Pozzetti} applies.

Endow $(\partial \bbB^n)^3$ with the measure class $[\mu]$ associated with the volume form on $\mathcal{C}^0_3$. We will show that the $[\mu]$-essential image of $\phi^3$ is contained in $\mathcal{C}^0_3$. In fact, we will show that it is contained in the subset $\mathcal{C}^1_3$ of $\calC^0_3$. Since $\mathcal{C}^1_3$ is conull in $\mathcal{C}^0_3$, we view $[\mu]$ as the measure class associated with the volume form on $\mathcal{C}^1_3$.

As indicated above, we identify $\mathcal{C}^1_3$ with the fibered square of the map $\al$ from $(\partial \bbB^n)^{(2)}$ to $\calC_1$, which is naturally a subset of $(\partial \bbB^n)^{(2)} \times (\partial \bbB^n)^{(2)}$, even though $\mathcal{C}^1_3$ is a subset of $(\partial \bbB^n)^{3}$ by its original definition. This is clarified by the commutative diagram
\[
\begin{tikzcd}
(\partial \bbB^n)^{(2)} \times_{\mathcal{C}_1} (\partial \bbB^n)^{(2)}
\arrow[hookrightarrow]{r} \arrow[leftrightarrow]{d}{\simeq} &
\left\{(x,y),\, (x,z) \in (\partial \bbB^n)^{(2)}\ :\ y,z \neq x \right\}
\arrow[leftrightarrow]{d}{\simeq} \\
\mathcal{C}^1_3
\arrow[hookrightarrow]{r} &
\left\{(x,y,z) \in (\partial \bbB^n)^3\ :\  y,z \neq x \right\}
\end{tikzcd}
\]
in which the top two spaces are subsets of $(\partial \bbB^n)^{(2)} \times (\partial \bbB^n)^{(2)}$ and the bottom two spaces are subsets of $(\partial \bbB^n)^3$ containing the support of $[\mu]$. Notice that this diagram also commutes with applying $\phi^{(2)}\times \phi^{(2)}$ to the top line and $\phi^3$ to the bottom line.

Further, the assumption that $\al \circ \phi^{(2)}$ and $\psi\circ \al$ agree almost everywhere allows us to apply Lemma~\ref{lem:fiberedproduct} with:
\begin{align*}
X&=Y=X'=Y'=(\partial \bbB^n)^{(2)} & f&=g=\phi^{(2)} \\
Z&=Z'=\mathcal{C}_1 & h&=\psi
\end{align*}
This allows us to conclude that the essential image of $\phi^{(2)}\times \phi^{(2)}$ is contained in $(\partial \bbB^n)^{(2)} \times_{\mathcal{C}_1} (\partial \bbB^n)^{(2)}$ with respect to the fibered measure class of $(\partial \bbB^n)^{(2)} \times_{\mathcal{C}_1} (\partial \bbB^n)^{(2)}$ on $(\partial \bbB^n)^{(2)} \times (\partial \bbB^n)^{(2)}$. In view of the above discussion, it follows that, with respect to the measure class associated with $\mathcal{C}^1_3$ on $(\partial \bbB^n)^{3}$, the essential image of $\phi^3$ is contained in $\mathcal{C}^1_3$. This completes the proof.
\end{proof}

\subsection{Proof of Proposition~\ref{prop:SRJZ}}\label{ssec:proofSRJS}

Let $G$ be $\SU(n,1)$ for $n \ge 2$ and $\Gam < G$ be a lattice. Suppose that $H=\PU(n,1)$, $\conj{Z}<H$ is the center of the unipotent radical of a proper parabolic subgroup $Q<H$, and $\rho:\Gamma\to H$ is a homomorphism with unbounded, Zariski dense image. Assume that there exists a continuous homomorphism $\tau:P\to Q/\conj{Z}$ with one-dimensional kernel and a measurable map $\Phi:G\to H/\conj{Z}$ that is $(P \times \Gam)$-equivariant with respect to the left $\Gamma$-action via $\rho$ and the right $P$-action through $\tau$. We must show that $\rho$ extends to a continuous homomorphism from $G$ to $H$.

We will freely use the notation introduced in \S\ref{subsec:subgroups}.  Identifying $Q$ with $P/C$ we consider the natural map $\theta:P \to P/CZ\cong Q/\conj{Z}$. For a fixed $p_0\in P$, the map $G\to H/\conj{Z}$ given by $g\mapsto \Phi(gp_0^{-1})$ is $(P \times \Gam)$-equivariant with respect to the left $\Gamma$-action and the right $P$-action via $\tau\circ\inn(p_0)$. We can then replace $\tau$ with $\tau\circ\inn(p_0)$, hence Proposition~\ref{prop:autP} allows us to assume that $\tau$ is surjective, $\ker(\tau)=CZ$, and $\tau(P\cap D)=\theta(P\cap D)$.

The composition of $\Phi:G\to H/\conj{Z}$ with $H/\conj{Z}\to H/Q$ is right $P$-invariant, and hence gives rise to a $\Gam$-map $P \bs G\to H/Q$. Upon identifying $P \bs G$ with $H/Q\simeq\partial \bbB^n$, we consider the above composition as a $\Gamma$-equivariant measurable map $\phi:\partial \bbB^n\to\partial \bbB^n$. By Corollary~\ref{cor:pozzetti}(1), the map $\phi^2:(\partial\bbB^n)^2\to(\partial\bbB^n)^2$ restricts to a map $\phi^{(2)}:(\partial\bbB^n)^{(2)}\to(\partial\bbB^n)^{(2)}$.

Next we consider the map $\al$ obtained by composing the natural identification $(\partial \bbB^n)^{(2)} \simeq \mathcal{C}_{(2)}$ with the forgetful map $\al_2:\mathcal{C}_{(2)}\to \mathcal{C}_1$. In other words, $\al$ maps the pair of distinct points $(x,y)$ to the pair consisting of the unique chain through $x$ and $y$ and the point $x$. Identify $MA \bs G$ with $(\partial \bbB^n)^{(2)}$ and $\mathcal{C}_1$ with $(H/\conj{Z})/\theta(P\cap D)$, where the latter identification comes from the fact that $\theta(P\cap D)=\tau(P\cap D)$.

Our goal is now to prove the existence of the dashed arrows in the following diagram of $(P \times \Gam)$-equivariant measurable maps:
\[
\begin{tikzcd}
G \arrow{d} & & & \\
MA \bs G \arrow{r}{\simeq} \arrow{d} & (\partial\bbB^n)^{(2)} \arrow{r}{\phi^{(2)}} \arrow{d}{\al} &
(\partial\bbB^n)^{(2)} \arrow{d}{\al} & \\
(P\cap D) \bs G \arrow{r}{\simeq} \arrow[dashed, bend right=15]{rrr}{\psi^\prime} & \mathcal{C}_1 \arrow[dashed]{r}{\psi} & \mathcal{C}_1 \arrow{r}{\simeq} & (H/\conj{Z})/\theta(P\cap D)
\end{tikzcd}
\]
For this, the composed map $G\to (H/\conj{Z})/\theta(P\cap D)$ is $(P\cap D)$-invariant, since $\theta(P\cap D)=\tau(P\cap D)$, thus it descends to a map from $(P\cap D) \bs G$ to $(H/\conj{Z})/\theta(P\cap D)$, which proves the existence of $\psi^\prime$. The existence of $\psi$ is then obtained by pre- and post-composing with the corresponding isomorphisms. We thus have that $\al \circ \phi^{(2)}$ and $\psi\circ \al$ agree almost everywhere, and therefore we can apply Corollary ~\ref{cor:pozzetti}(2) and conclude that $\rho$ extends to a continuous homomorphism from $G$ to $H$. This proves the proposition. \qed

\begin{rem}
The use of incidence geometry to prove rigidity theorems goes back at least to Mostow's use of work of Tits in his proof of Mostow Rigidity \cite{Mostow}. Rigidity of chain preserving maps is older, going back to Cartan's 1932 paper \cite{Cartan}. An important application of rigidity of chain preserving maps in the study of representations of discrete groups was by Burger and Iozzi in \cite{BurgerIozzi}. The idea of exploiting triples of points on a chain goes back to earlier work of Toledo on rigidity of certain surface group representations into $\PU(n,1)$ \cite{Toledo}, and Toledo attributes this general idea to Thurston. As mentioned above, we cannot use Cartan's result because our boundary map is only measurable and we cannot use Burger--Iozzi's because they assume orientability of the map on chains, but Pozzetti chain rigidity theorem suffices for our purposes.
\end{rem}

\begin{rem}\label{rem:SOremark}
Margulis and Mohammadi proved a version of Theorem~\ref{thm:main} for cocompact lattices in $\SO(3,1)$ using incidence geometry. It is possible to use the methods of this section combined with results in \cite{BFMS} to produce the input for their incidence geometry result, i.e., to prove the existence of a \emph{circle preserving map} between boundaries. However, the proof of Theorem~\ref{thm:main} in the real hyperbolic case is easier to complete using compatibility as we did in \cite{BFMS}. More generally though, it is possible for one to apply Lemma~\ref{lem:Spar} above and use the methods of this paper to turn some problems about homomorphisms from lattices in $\SO_0(n,1)$ to incompatible targets into incidence geometry problems for boundary maps.
\end{rem}

\section{Algebraic groups associated with lattices in $\SU(n,1)$ and Theorem~\ref{thm:SRht}}\label{sec:Hodge}

We begin by establishing some notation we need for the proof of Theorem~\ref{thm:SRht}. Let $\Gam < \SU(n,1)$ be a lattice, $n \ge 2$. Following the discussion of \cite[\S3]{BFMS}, local rigidity (due to Calabi--Vesentini \cite{Calabi-Vesentini} in the cocompact case and Raghunathan \cite{RaghunathanRigid} otherwise) along with work of Vinberg \cite{VinbergDef} implies that the field $\ell = \Tr(\Ad(\Gam))$ is both a number field and a minimal field of definition for $\Gam$. Let $\conj{\bfG}$ denote the connected adjoint $\ell$-algebraic group defined by the Zariski closure of $\Gam$ under the adjoint representation.

Then $\conj{\bfG}$ is the adjoint form of a unique simply connected $\ell$-group $\bfG$ of type ${}^{2}\mathrm{A}_n$ \cite{TitsSS}. In other words, there is a quadratic extension $\ell^\prime / \ell$ and a central division algebra $D$ over $\ell^\prime$ of degree $d \mid (n+1)$ with involution $\sig$ of the second kind so that $\bfG$ is the special unitary group $\SU_{(n+1)/d}(D, h)$ for some nondegenerate $\sig$-hermitian form $h$ on $D^{(n+1)/d}$. Recall that $\sig$ is of the second kind if its restriction to the center of $D$ is the nontrivial Galois involution of $\ell^\prime / \ell$. The image of $\Gam$ under the adjoint representation lies in $\conj{\bfG}(\ell)$ and the kernel of $\SU(n,1) \to \PU(n,1)$ is cyclic of order $(n+1)$, hence we can replace $\Gam$ with a subgroup of finite index and assume that $\Gam < \bfG(\ell)$.

\medskip

\noindent
With this setup, we now prove Theorem~\ref{thm:SRht}.

\begin{proof}[Proof of Theorem~\ref{thm:SRht}]
Let $\Gam < \SU(n,1)$ be a lattice, and retain all previous definitions and notation from this section. As described above, we can assume that $\Gam < \bfG(\ell)$. For each place $v$ of $\ell$, we obtain a homomorphism
\[
\rho_v : \Gam \hookrightarrow G_v = \bfG(\ell_v),
\]
where $\ell_v$ is the completion of $\ell$ associated with $v$. Let $v_0$ be the place associated with the lattice embedding of $\Gam$ into $\SU(n,1)$.

Our first goal is to prove that $\rho_v$ is locally rigid for all $v$. Recall that local rigidity of $\rho_v$ is the vanishing of the cohomology group $H^1(\rho_v, \frakg_v)$, where $\frakg_v$ is the Lie algebra of $G_v$. The assumption that $\Gam < \bfG(\ell)$ implies that there is an $\ell$-form $\frakg_\ell$ of $\fraks\fraku(n,1)$ so that $\frakg_v \cong \frakg_\ell \otimes_\ell \ell_v$ for all $v$. Note that vanishing of $H^1(\rho_v, \frakg_v)$ is equivalent to vanishing of $H^1(\Gam, \frakg_\ell)$ since
\[
H^1(\rho_v, \frakg_v) \cong H^1(\Gam, \frakg_\ell) \otimes_\ell \ell_v.
\]
However, $H^1(\Gam, \frakg_\ell)$ is trivial since $H^1(\rho_{v_0}, \frakg_{v_0})$ is trivial by local rigidity of the lattice embedding of $\Gam$ into $\SU(n,1)$. This proves local rigidity of $\rho_v$ for all $v$.

We now make some technical reductions needed to apply some results as they are stated in the literature. First, without loss of generality we can pass to a finite index torsion-free subgroup of $\Gam$, since $\ell$, $\ell^\prime$, and $\bfG$ are commensurability invariants by \cite{VinbergDef}. When $\Gam$ is cocompact, we can therefore assume that $\bbB^n / \Gam$ is a compact K\"ahler manifold. When $\Gam$ is not cocompact, we can assume that $\bbB^n / \Gam$ is a smooth quasiprojective variety admitting a smooth toroidal compactification by a smooth divisor $D$ (e.g., see \cite{AMRT, Mok}).

More precisely in the noncompact case, we can assume by passing to a finite index subgroup that the cusps of $\bbB^n / \Gam$ are diffeomorphic to bundles over the punctured disk with fiber an abelian variety. The associated peripheral subgroup of $\Gam$ is a two-step nilpotent group with infinite cyclic center naturally realized as a torsion-free lattice in the unipotent radical of the Borel subgroup of $\PU(n,1)$ (cf.\ the structure of $U$ in \S\ref{subsec:subgroups}). The center of this nilpotent group is generated by the free homotopy class of a loop projecting to a loop around the puncture of the disk in the base of the bundle. One can then smoothly compactify $\bbB^n / \Gam$ to obtain a smooth projective variety by adding a certain abelian variety above the puncture in the disk. See Figure~\ref{fig:STC} and see \cite[\S4.2]{Holzapfel} for further details in this specific case.

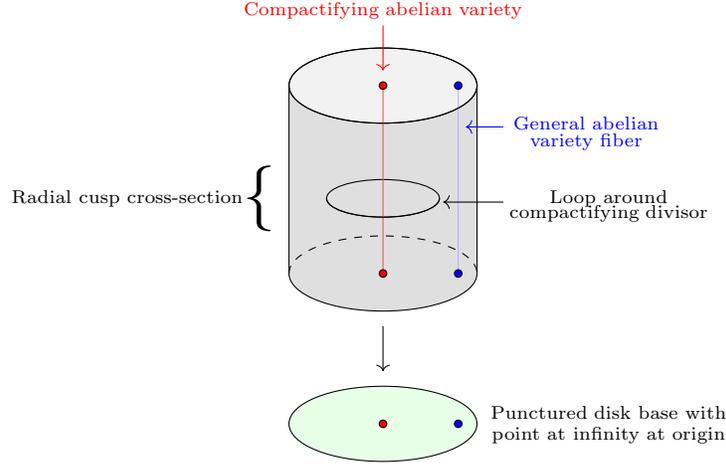
\begin{figure}
\begin{tikzpicture}
\fill [gray!25] (-1.25,0) -- (-1.25,-2.5) arc (180:360:1.25 and 0.5) -- (1.25,0) arc (0:180:1.25 and -0.5);
\draw [fill = gray!10] (0,0) ellipse (1.25 and 0.5);
\draw (0,-1.5) ellipse (0.75 and 0.25);
\draw (-1.25,0) -- (-1.25,-2.5);
\draw (-1.25,-2.5) arc (180:360:1.25 and 0.5);
\draw [dashed] (-1.25,-2.5) arc (180:360:1.25 and -0.5);
\draw (1.25,-2.5) -- (1.25,0);
\draw [fill = green!10] (0,-4.5) ellipse (1.25 and 0.5);
\draw [fill = red] (0,-4.5) circle (0.05);
\draw [red!70] (0,-2.5) -- (0,0);
\draw [fill = red] (0,-2.5) circle (0.05);
\draw [fill = red] (0,0) circle (0.05);
\draw [blue!30] (1,-2.5) -- (1,0);
\draw [fill = blue] (1,-2.5) circle (0.05);
\draw [fill = blue] (1,0) circle (0.05);
\draw [fill = blue] (1,-4.5) circle (0.05);
\draw (0,0) ellipse (1.25 and 0.5);
\draw [->] (0, -3.2) -- (0, -3.8);
\node at (3, -4.35) {{\tiny Punctured disk base with}};
\node at (3, -4.65) {{\tiny point at infinity at origin}};
\node at (-3.4, -1.5) {{\tiny Radial cusp cross-section}};
\node at (3, -1.5) {{\tiny Loop around}};
\node at (3, -1.7) {{\tiny compactifying divisor}};
\draw [->] (1.6, -1.55) -- (0.8, -1.55);
\node [blue] at (2.7, -0.5) {{\tiny General abelian}};
\node [blue] at (2.7, -0.75) {{\tiny variety fiber}};
\draw [blue, ->] (1.6, -0.55) -- (1.1, -0.55);
\draw [red, ->] (0, 0.8) -- (0, 0.2);
\node [red] at (0, 1) {{\tiny Compactifying abelian variety}};
\node at (-1.65,-1.5) {\Huge$\{$};
\draw (-0.75,-1.5) arc(180:360:0.75 and 0.25);
\end{tikzpicture}
\caption{Smooth compactification of a cusp.}\label{fig:STC}
\end{figure}

Suppose that $v$ is archimedean, i.e., that $\ell_v$ is $\bbR$ or $\bbC$. Rigidity of $\rho_v$ then implies that the real Zariski closure of $\rho_v(\Gam)$, namely  $G_v$, is a group of \emph{Hodge type}. See \cite[Lem.\ 4.5]{Simpson} when $\Gam$ is cocompact. In general see \cite[\S10.2.3]{Mochizuki}, which handles representations of fundamental groups of smooth quasiprojective varieties admitting a smooth compactification, and hence by the above assumptions applies to $\bbB^n / \Gam$. Since $\bfG$ is of absolute type $\mathrm{A}_n$, considering \cite[p.\ 50-51]{Simpson}, we conclude that $v$ is real and $G_v$ must be $\SU(r_v, s_v)$ for an appropriate pair $r_v, s_v$, completing the proof of the first two parts of the theorem.

It remains to prove that $\Gam$ is integral. Indeed, the last statement of Theorem~\ref{thm:SRht} is an immediate consequence of the previous statements along with the definition of arithmeticity. Integrality will follow from a theorem of Esnault--Groechenig \cite[Thm.\ 1.1]{EsnaultGroechenig} once we verify that their assumptions hold in our setting.

Since $\rho_v$ is locally rigid and has determinant one, it remains to verify that $\rho_v$ has quasi-unipotent monodromy at infinity. See \cite[\S2]{EsnaultGroechenig} for the precise definition. Indeed, loops around the compactifying divisor of $\bbB^n / \Gam$ are associated with central elements of peripheral subgroups of $\Gam$, which are naturally unipotent subgroups of $\bfG(\ell)$. These elements clearly remain unipotent under $\rho_v$, hence $\rho_v$ has quasi-unipotent monodromy at infinity. We then conclude from \cite[Thm.\ 1.1]{EsnaultGroechenig} that $\Gam$ is integral, which completes the proof of the theorem.
\end{proof}

\begin{rem}
With notation as in Theorem~\ref{thm:SRht}, if $\Gam < \PU(n,1)$ is a nonarithmetic lattice, there is another real place $v$ so that $\bfG(\ell_v) \cong \PU(p,q)$ for some $p + q = n+1$ and $p \ge q$. When $n = 2$, the only possibility is $(p,q) = (2,1)$. There are only two known commensurability classes of nonarithmetic lattices, and in each case $(p,q) = (3,1)$ for every noncompact $\bfG(\ell_v)$ \cite{Deraux3d}. We do not know if it is possible to have a nonarithmetic lattice with $\bfG(\ell_v) \cong \PU(p,q)$ with $q \ge 2$.
\end{rem}

\section{The proof of Theorem~\ref{thm:main}}\label{sec:mainpf}

\subsection{Equidistribution on $G / \Gam$}\label{ssec:Equi}

This subsection describes the equidistribution results needed to show that Theorem~\ref{thm:SR} applies to prove Theorem~\ref{thm:main}. Each result in this section has a direct analogue in \cite{BFMS}, and here we provide the necessary modifications for the $\SU(n,1)$ setting.

Throughout this section we fix a lattice $\Gam < G$. Recall that a measure $\mu$ on $G / \Gam$ is called \emph{homogeneous} if there exists a closed subgroup $S \le G$ and a closed $S$-orbit in $G/\Gamma$ such that $\mu$ is the push-forward of Haar measure on $S$ along this orbit. If $W\le S$ is a closed subgroup with respect to which $\mu$ is ergodic, we call the measure \emph{$W$-ergodic}. Given a homogeneous, $W$-ergodic measure $\mu$, we will refer to its support, $\supp(\mu)$, as a \emph{homogeneous, $W$-ergodic subspace} of $G/\Gamma$.

We ask the reader to recall the notation introduced in \S\ref{subsec:simple}. In particular, we have a fixed maximal compact subgroup $K\le G$ and we consider the symmetric space $K\bs G$. The locally symmetric space $K \bs G / \Gam$ will be denoted by $M$. In Definition~\ref{def:standard} we introduced the standard (almost simple) subgroups of $G$, namely certain copies in $G$ of the group $\SU(m,1)$ for $1\le m\le n$ and $\SO_0(m,1)$ for $2\le m\le n$ which, by Proposition~\ref{prop:almostsimple}, are representatives of conjugacy classes of all noncompact, connected, almost simple, closed subgroups of $G$.

In what follows ``geodesic subspace'' will always mean a properly immersed totally geodesic subspace of either $M$ or its universal cover $K\bs G$. Our goal in this subsection is to prove the following proposition, which translates the existence of infinitely many maximal geodesic subspaces of $M$ into a statement about measures on $G/\Gam$ that are invariant under a standard subgroup of $G$.

\begin{prop}[Cf.\ Prop.\ 3.1 in \cite{BFMS}]\label{prop:BFMS3.1}
The following are equivalent:
\begin{enumerate}

\item The complex hyperbolic space $M=K \bs G / \Gam$ contains infinitely many maximal geodesic subspaces of real dimension at least $2$.

\item There exists a standard subgroup $W< G$ and an infinite sequence $\{\mu_i\}$ of $W$-invariant, $W$-ergodic measures on $G / \Gam$ with proper support for which Haar measure on $G / \Gam$ is a weak-$*$ limit of the $\mu_i$.

\item There exists a standard subgroup $W<G$ and an infinite sequence $\{\mu_i\}$ of homogeneous, $W$-ergodic measures on $G / \Gam$ for which Haar measure on $G / \Gam$ is a weak-$*$ limit of the $\mu_i$.
\end{enumerate}
\end{prop}

As noted in \cite{BFMS}, work of Ratner \cite{RatnerDuke} shows that (2) and (3) are equivalent, and therefore it remains to prove that (1) and (3) are equivalent. For this, we closely follow the strategy in \cite[\S 3.1]{BFMS}, adapting the statements therein to the current setting.

For the remainder of this section, fix the notation $\pi:G/\Gamma\to K\bs G/\Gamma=M$ for the natural projection map. Note that $\pi$ is a proper map. We begin with the following lemma.

\begin{lemma}\label{lem:TGtomsr}
The following hold:
\begin{enumerate}

\item Let $S$ be a closed subgroup of $G$ such that $W \le S$ for some standard subgroup $W < G$, and suppose there exists $h\in G$ for which $Sh\Gamma/\Gamma$ is a closed $S$-orbit. Then $Z=\pi(Sh\Gamma/\Gamma)$ is a closed totally geodesic subspace of $M$ of real dimension at least $2$. Up to normalization, the $\dim(Z)$-volume of $Z$ is the push-forward of Haar measure on $Sh\Gamma/\Gamma$ via the projection map $\pi$. \label{num:TGtomsr1}

\item Under the assumptions of part \eqref{num:TGtomsr1}, $Z=M$ if and only if $S=G$ and $Z=\pi(Wg\Gamma/\Gamma)$ for some $g\in G$ if and only if $W \le S\le N$, where $N$ is the normalizer of $W$ in $G$. In the latter case, $Nh\Gamma/\Gamma$ is also closed with projection $\pi(Nh\Gamma/\Gamma)=Z$. \label{num:TGtomsr2}

\item Conversely, every totally geodesic subspace $Z$ in $M$ of real dimension at least $2$ has finite measure. Moreover, for any such $Z$ there is a standard subgroup $W$ of $G$ with normalizer $N$, an intermediate subgroup $W \le S \le N$, and an element $h\in G$ such that $\Gam \cap S^{h^{-1}}$ is a lattice in $S^{h^{-1}}$ and $S h \Gamma/\Gamma$ is a homogeneous, $W$-ergodic subspace of $G/\Gamma$ for which $Z=\pi(W h\Gamma/\Gamma)=\pi(S h\Gamma/\Gamma)$.\label{num:TGtomsr3}

\end{enumerate}
\end{lemma}

\begin{proof}
The fact that $Z$ is closed follows from the fact that $\pi$ is a proper map and the assumption that $Sh\Gamma/\Gamma$ is closed. By Lemma \ref{lem:AlignedToTG}\eqref{lem:AlignedToTG3}, $K\bs KSh$ is a totally geodesic subspace of $K\bs G$, thus $Z$, which is the image of $K \bs KS$ under the covering map $K\bs G \to M$, is a totally geodesic subspace of $M$. The statement about the measure also follows from Lemma \ref{lem:AlignedToTG}\eqref{lem:AlignedToTG3}, and this proves part \eqref{num:TGtomsr1}.

Next we consider part \eqref{num:TGtomsr2}. Clearly $S=G$ implies $Z=M$ and the fact that $Z=M$ implies $S=G$ follows by applying Lemma \ref{lem:AlignedToTG}\eqref{lem:AlignedToTG1} in the special case $W=G$. If $W \le S\le N$ then
\[
Z=\pi(Sh\Gamma/\Gamma)=\pi(Wh\Gamma/\Gamma),
\]
by Lemma \ref{lem:AlignedToTG}\eqref{lem:AlignedToTG4}. Conversely, assume
$Z=\pi(Wg\Gamma/\Gamma)$ for some $g\in G$. Then both $K\bs KSh$ and $K\bs KWg$ cover $Z$ in $K\bs G$, thus they are totally geodesic subspaces of the same type. Therefore $K\bs KSh=K\bs KWg_0$ for some $g_0\in G$ by Proposition~\ref{prop:standard}, and Lemma \ref{lem:AlignedToTG}\eqref{lem:AlignedToTG3} implies that $K\bs KSh=K\bs KS^+h$. Hence we obtain the sequence of equalities
\[
K\bs Khg_0^{-1}(S^+)^{g_0h^{-1}}=K\bs KS^+hg_0^{-1} =K\bs KShg_0^{-1} = K\bs KW,
\]
and conclude that $(S^+)^{g_0h^{-1}}$ preserves $K\bs KW$ and acts transitively on it. Thus by Lemma \ref{lem:AlignedToTG}\eqref{lem:AlignedToTG1} we see that $(S^+)^{g_0h^{-1}}=W$.

Since $W\le S$, we have $(S^+)^{g_0h^{-1}}=W=W^+\le S^+$ and deduce that $S^+=(S^+)^{g_0h^{-1}}=W$. Since $S$ normalizes $S^+$ it follows that $S\le N$. Then $W$ is cocompact in $N$, so $S$ is cocompact in $N$ by Lemma \ref{lem:AlignedToTG}\eqref{lem:AlignedToTG2}. It follows that $Nh\Gamma/\Gamma$ is closed in $G/\Gamma$, since $Sh\Gamma/\Gamma$ is closed by hypothesis. That $\pi(Nh\Gamma/\Gamma)=Z$ follows from Lemma \ref{lem:AlignedToTG}\eqref{lem:AlignedToTG4}, and this completes the proof of part \eqref{num:TGtomsr2}.

We now prove part \eqref{num:TGtomsr3}. Fix a totally geodesic subspace $Z$ in $M$ of real dimension at least $2$. The fact that $Z$ has a finite volume is a well-known consequence of the existence of a thick-thin decomposition. See \cite{GeLev} for detailed argument in real hyperbolic space that is easily adapted to any rank one symmetric space.

The preimage of $Z$ under the $\Gamma$-invariant map $K\bs G \to M$ is a collection of mutually disjoint totally geodesic subspaces on which $\Gamma$ acts. We fix one of these, which by Proposition~\ref{prop:standard} is of the form $K\bs KWg$ for some $g\in G$ and a standard subgroup $W\le G$. By Lemma \ref{lem:AlignedToTG}\eqref{lem:AlignedToTG2}, the subgroup of $G$ that stabilizes this totally geodesic subspace is $N^{g^{-1}}$, thus the subgroup of $\Gamma$ stabilizing it is $\Gamma \cap N^{g^{-1}}$. By Lemma \ref{lem:AlignedToTG}\eqref{lem:AlignedToTG3} we have that $K\bs KWg=K\bs KNg=K\bs KgN^{g^{-1}}$ which is isomorphic as an $N^{g^{-1}}$-space to $(K^{g^{-1}}\cap N^{g^{-1}}) \bs N^{g^{-1}}$. It follows that
\[
Z=K\bs KgN^{g^{-1}}\Gamma/ \Gamma \simeq (K^{g^{-1}}\cap N^{g^{-1}}) \bs N^{g^{-1}}/ (\Gamma \cap N^{g^{-1}}).
\]
Since $Z$ has a finite volume and $K^{g^{-1}}\cap N^{g^{-1}}$ is a compact subgroup of $N^{g^{-1}}$, we conclude that $\Gamma \cap N^{g^{-1}}$ is a lattice in $N^{g^{-1}}$. By \cite[Thm.~1.13]{Rag} we obtain that $Ng\Gamma/\Gamma$ is closed in $G/\Gamma$ and conclude that it is a closed homogeneous $N$-orbit of finite volume for which the associated Haar measure on this orbit is a homogeneous measure.

However $N g \Gamma/\Gamma$ may not be $W$-ergodic, even though it is $W$-invariant. To complete the proof, it remains to show that there exists an intermediate subgroup $W \le S\le N$ and an $h\in G$ such that $S h\Gamma/\Gamma$ is a homogeneous, $W$-ergodic subspace of $G/\Gamma$. Let $\mu$ be a $W$-ergodic measure in the ergodic decomposition of Haar measure on $N g\Gamma/\Gamma$, and let $S$ denote the stabilizer of $\mu$ in $N$. Then $W\le S\le N$ by $W$-invariance, and Ratner's theorem implies that $\mu$ is $S$-homogeneous. Write the corresponding closed homogeneous space as $S h\Gamma/\Gamma$ for some $h\in G$. Since $S h\Gamma/\Gamma\simeq S/(\Gam \cap S^{h^{-1}})$ we see that $\Gam \cap S^{h^{-1}}$ is a lattice in $S^{h^{-1}}$. Since $h$ is in the $N$-homogeneous space $Ng\Gamma/\Gamma$, we have that $Nh\Gamma/\Gamma=Ng\Gamma/\Gamma$. Using Lemma \ref{lem:AlignedToTG}\eqref{lem:AlignedToTG3} twice, we see that
\[
\pi(Wh\Gamma/\Gamma)=\pi(Sh\Gamma/\Gamma)=\pi(Nh\Gamma/\Gamma)=\pi(Ng\Gamma/\Gamma)=\pi(Wg\Gamma/\Gamma)=Z,
\]
which completes the proof.
\end{proof}

We now collect some useful facts about limits of homogeneous, $W$-ergodic measures that we need in the proof of Proposition~\ref{prop:BFMS3.1}. Despite the difference in presentation, the next three results are direct analogues of results in the real hyperbolic case \cite[Thm 3.3]{BFMS} and use an essentially identical argument. All of these results are relatively straightforward consequences of work of Dani--Margulis \cite[Thm.\ 6.1]{DaniMargulis} and Mozes--Shah \cite{MozesShah}. We first show that there is no escape of mass for the sequence of measures under consideration.

\begin{lemma}\label{lem:nomassescape}
Let $W<G$ be a closed, connected, almost simple subgroup of $G$ that is generated by unipotent elements. If $\{\mu_i\}$ is a sequence of homogeneous, $W$-ergodic probability measures on $G/\Gamma$ that weak-$*$ converges to a measure $\mu$ in the space of all finite Radon measures, then $\mu$ is not the zero measure.
\end{lemma}

\begin{proof}
Since the claim is invariant under conjugation, we can assume by Proposition~\ref{prop:almostsimple} that $W$ is a standard subgroup of $G$.

Let $C$ be any fixed compact set in $M$ whose interior contains the compact core of $M$. See \cite[Thm.\ 10.5]{BGS} for the existence of a compact core in this setting. Then $C$ has the property that $C\cap Z$ contains a nonempty open subset of $Z$ for any closed geodesic subspace $Z$ of $M$. Set $F=\pi^{-1}(C)$.

Applying \cite[Thm.\ 6.1]{DaniMargulis} for $\epsilon=1/2$, there exists a compact set $F^\prime \subseteq G/\Gamma$ such that
\begin{equation}\label{eqn:DM}
\frac{1}{2}\le \frac{1}{T}\lambda\left(\left\{t\in[0,T]\mid u_tx\in F^\prime\right\}\right)=\frac{1}{T}\int_0^T\chi_{F^\prime}(u_t x)dt,
\end{equation}
for every one-parameter unipotent subgroup $\{u_t\}$ of $G$, every $x\in F$, and every $T\ge 0$, where $\chi_{F^\prime}$ is the characteristic function of $F^\prime$ and $\lambda$ is Lebesgue measure on $\R$. We claim that $\mu(F^\prime)\ge 1/2$, which proves that $\mu$ is not the zero measure.

Fix a one-parameter unipotent subgroup $U=\{u_t\}$ of $W$. Then $W$-ergodicity and the Howe--Moore theorem imply that each $\mu_i$ is $U$-ergodic. The measures $\mu_i$ are homogeneous and $W$-ergodic, so Lemma \ref{lem:TGtomsr}\eqref{num:TGtomsr1} implies that for each $i$ there exists a closed totally geodesic subspace $Z_i$ of $M$ with real dimension at least $2$ such that $\pi_*\mu_i$ is a constant multiple of the volume measure on $Z_i$. Since $C\cap Z_i$ contains an open subset of $Z_i$, we see that
\[
\pi_*\mu_i(C)=\mu_i(\pi^{-1}(C))>0.
\]
In particular, there exists a $U$-generic point $x_i\in F$ for each $\mu_i$. The Birkhoff ergodic theorem applied to $\chi_{F^\prime}$ combined with Equation \eqref{eqn:DM} then implies that $\mu_i(F^\prime)\ge 1/2$. Therefore, $\mu(F^\prime)\ge 1/2$, which completes the proof.
\end{proof}

\begin{cor}\label{cor:nomassescape}
Under the assumptions of Lemma \ref{lem:nomassescape}, $\mu$ is a homogeneous, $W$-ergodic probability measure on $G/\Gamma$. Moreover, $\mu$ is ergodic with respect to any nontrivial subgroup generated by unipotent elements of its stabilizer in $G$.
\end{cor}

\begin{proof}
Let $S$ denote the stabilizer of $\mu$ in $G$. Then $\mu$ is not the zero measure by Lemma~\ref{lem:nomassescape}, thus \cite[Cor.\ 1.3]{MozesShah} shows that $\mu$ is a homogeneous, $S^+$-ergodic probability measure on $G/\Gamma$. Since any nontrivial subgroup generated by unipotents in $S$ is a noncompact subgroup of $S^+$, $\mu$ is ergodic with respect to any such subgroup by the Howe--Moore theorem. In particular, this applies to $W$, as $W\le S$ because the set of $W$-invariant measures is weak-$*$ closed.
\end{proof}

We now apply \cite[Thm.\ 1.1]{MozesShah} to understand the relationship between the support of the $\mu_i$ and the support of $\mu$.

\begin{thm}\label{thm:MozesShah}
Let $W<G$ be a closed, connected, almost simple subgroup generated by unipotent elements and $\{\mu_i\}$ be a sequence of homogeneous, $W$-ergodic measures on $G/\Gamma$ that weak-$*$ converges to $\mu$. Then there exist a sequence of elements $\{g_i\}$ in $G$ and a natural number $i_0$ such that for all $i\ge i_0$, the measures $g_i\mu$ are homogeneous, $W$-ergodic probability measures on $G/\Gamma$ with $\supp(\mu_i)\subseteq g_i\supp(\mu)$.
\end{thm}

\begin{proof}
We fix unipotent subgroups $U_1,\dots, U_s$ that generate $W$. By the Howe--Moore theorem, $\mu_i$ is $U_j$-ergodic for every $i\in\N$ and every $1\le j\le s$. Therefore, for each $i$, the set of points for which $\mu_i$ is $U_j$-generic for all $1\le j\le s$ is of full measure and hence is dense in $\supp(\mu_i)$.

By Corollary \ref{cor:nomassescape}, $\mu$ is a nonzero homogeneous $W$-ergodic probability measure and $\supp(\mu)$ is a nonempty, closed homogeneous subspace of $G/\Gamma$. Fix a point $x_\infty$ in $\supp(\mu)$. Then by the above we can find a sequence of points $\{x_i\}$ converging to $x_\infty$ such that, for each $i$, $x_i\in \supp(\mu_i)$ and $x_i$ is a $U_j$-generic point for every $1\le j\le s$. We also fix a sequence $\{g_i\}$ of elements of $G$ converging to the identity such that $g_ix_\infty=x_i$ for each $i$.

Applying \cite[Thm.\ 1.1]{MozesShah} to each unipotent subgroup $U_j$, we find a natural number $i_j$ such that for all $i\ge i_j$ the measure $\mu$ is $U_j^{g_i^{-1}}$-invariant and $\supp(\mu_i)\subseteq g_i\supp(\mu)$. Let $i_0=\max\{i_1,\dots,i_s\}$. We now have that for every $i\geq i_0$, $\supp(\mu_i)\subseteq g_i\supp(\mu)$ and we are left to show that for every such $i$, the measure $g_i\mu$, which is clearly a homogeneous probability measure, is in fact $W$-ergodic. Equivalently, we are left to show that for every $i\geq i_0$, the measure $\mu$ is $W^{g_i^{-1}}$-ergodic.

Fix $i\ge i_0$. We have that $\mu$ is $U_j^{g_i^{-1}}$-invariant for every $1\le j\le s$. We conclude that $\mu$ is $W^{g_i^{-1}}$-invariant, as $W^{g_i^{-1}}$ is the group generated by the unipotent subgroups $U_j^{g_i^{-1}}$, $1\le j\le s$. It follows from Corollary \ref{cor:nomassescape} that indeed $\mu$ is $W^{g_i^{-1}}$-ergodic, which completes the proof of the theorem.
\end{proof}

We finally have all of the necessary ingredients to prove Proposition \ref{prop:BFMS3.1}, which has an essentially identical proof to that of \cite[Prop.\ 3.1]{BFMS}. As remarked immediately following its statement, it suffices to show that (1) and (3) are equivalent.

\begin{proof}[Proof of Proposition \ref{prop:BFMS3.1}]
We first prove that (3) implies (1). Recall that $\{\mu_i\}$ is a sequence of homogeneous, $W$-ergodic measures with weak-$*$ limit $\mu$, which is Haar measure on $G/\Gamma$. Then $\pi_*\mu$ is the volume measure on $M$ and hence $\supp(\pi_* \mu)=M$. As the measures $\mu_i$ are homogeneous, $\pi_*\mu_i$ is a constant multiple of the volume form on some closed totally geodesic subspace of $M$ by Lemma \ref{lem:TGtomsr}\eqref{num:TGtomsr1}, therefore $\supp(\pi_*\mu_i)$ is contained in some closed maximal geodesic subspace $Z_i$ of $M$. Then $M=\supp(\pi_*\mu)$ is contained in closure of $\cup Z_i$, and we conclude that there must be infinitely many distinct $Z_i$ in this union. This implies (1).

Finally, we show that (1) implies (3). Assume there are infinitely many distinct closed maximal totally geodesic subspaces $\{Z_i\}$ of $M$. By Lemma \ref{lem:TGtomsr}\eqref{num:TGtomsr3}, for every $i$ there exists a standard subgroup $W_i\le G$, a homogeneous $W_i$-ergodic measure $\mu_i$ on $G/\Gamma$, and an element $h_i\in G$ such that $Z_i=\pi(W_i h_i\Gamma/\Gamma)=\pi(X_i)$, where $X_i=\supp(\mu_i)$. Since the collection of standard subgroups of $G$ is finite, we may and do pass to a subsequence for which the subgroups $W_i$ all coincide. We denote this common standard subgroup by $W$. Upon passing to a further subsequence, which we still denote by $\{\mu_i\}$, we assume that the $\mu_i$ weak-$*$ converge to a probability measure $\mu$ that is homogeneous and $W$-ergodic by Corollary \ref{cor:nomassescape}. The proof will be complete once we show that $\mu$ is Haar measure on $G/\Gamma$.

For a contradiction, assume $\mu$ is not Haar measure. Since $\mu$ is homogeneous, $X_\infty=\supp(\mu)$ is a closed homogeneous subspace $Sh\Gamma/\Gamma$ of $G/\Gamma$, where $h\in G$ and $S$ is the stabilizer of $\mu$. By the assumption that $\mu$ is not Haar measure on $G/\Gamma$, $S<G$ is a proper subgroup. By Theorem \ref{thm:MozesShah}, there exists a sequence of elements $\{g_i\}$ in $G$ and a natural number $i_0$ such that for all $i\ge i_0$, the measures $g_i\mu$ are homogeneous, $W$-ergodic probability measures on $G/\Gamma$ with $\supp(\mu_i)\subseteq g_i\supp(\mu)$. Upon passing again to a subsequence, we assume that $i_0=1$, thus for every $i$, $X_i\subset g_i X_\infty$ and $g_i\mu$ is a homogeneous, $W$-ergodic probability measure on $G/\Gamma$.

We will find a contradiction by showing that the spaces $Z_i$ all coincide with $\pi(X_\infty)$, contradicting the assumption that they are all distinct. From now on we fix a natural number $i$ and will argue that $Z_i=\pi(X_\infty)$.

By Lemma \ref{lem:TGtomsr}\eqref{num:TGtomsr1}, $\pi(g_i X_\infty)$ is a totally geodesic subspace of $M$, and by Lemma \ref{lem:TGtomsr}\eqref{num:TGtomsr2} it is a proper subspace of $M$, since the stabilizer of $g_i\mu$, namely $S^{g_i}$, is a proper subgroup of $G$. From $X_i\subseteq g_i X_\infty$ we have that $Z_i=\pi(X_i)\subseteq\pi(g_i X_\infty)$ and, by the maximality assumption on the totally geodesic subspace $Z_i$, we conclude that $Z_i=\pi(g_i X_\infty)$. Since $g_i\mu$ is $W$-invariant and its stabilizer is $S^{g_i}$, we have that $W\le S^{g_i}$. We also have $\pi(S^{g_i}g_ih\Gam/\Gam)=\pi(W h_i\Gamma/\Gamma)$, by the sequence of equations
\[
\pi(S^{g_i}g_ih\Gam/\Gam)=\pi(g_iSh\Gam/\Gam)=\pi(g_i X_\infty)=Z_i=\pi(W h_i\Gamma/\Gamma).
\]
Lemma \ref{lem:TGtomsr}\eqref{num:TGtomsr2} implies that $W \le S^{g_i}\le N$, where $N$ is the normalizer of $W$.

It follows from Lemma~\ref{lem:AlignedToTG}\eqref{lem:AlignedToTG4} that $W=(S^{g_i})^+$. As $W\le S$, we also have that $W^{g_i}\le (S^{g_i})^+=W$, hence $W^{g_i}=W$ and in particular, $g_i\in N$. Using Lemma \ref{lem:TGtomsr}\eqref{num:TGtomsr2} again, we conclude that $N g_ih\Gamma/\Gamma$ is a closed homogeneous subspace of $G/\Gamma$ such that $\pi(S^{g_i}g_ih\Gam/\Gam)=\pi(N g_ih\Gamma/\Gamma)$. Conjugating by $g^{-1} \in N$, the inclusion $W \le S^{g_i}\le N$ implies that $W \le S \le N$ and another application of Lemma \ref{lem:TGtomsr}\eqref{num:TGtomsr2} gives $\pi(X_\infty)=\pi(Sh\Gam/\Gam)=\pi(N h\Gamma/\Gamma)$. Therefore,
\[
Z_i=\pi(g_iX_\infty)=\pi(S^{g_i}g_ih\Gam/\Gam)=\pi(N g_ih\Gamma/\Gamma)=\pi(N h\Gamma/\Gamma)=\pi(X_\infty).
\]
Thus $Z_i=\pi(X_\infty)$, and this completes the proof that $\mu$ is Haar measure, hence we have shown that (1) implies (3).
\end{proof}

\begin{rem}
It is also worthy of mention here that there has been considerable previous work on equidistribution in the context of Shimura varieties and special subvarieties. For example, see deep work of Clozel--Ullmo on equidistribution for strongly special subvarieties \cite{ClozelUllmo}. Also see \cite{Ullmo, Yafaev, UllmoYafaev, MollerToledo, KoziarzMaubon2} for subsequent results in this direction.
\end{rem}

\subsection{Setup for the proof of Theorem~\ref{thm:main}}\label{ssec:arithpfsetup}

The purpose of this section is to collect the necessary final setup so that we can apply Theorem \ref{thm:SR} in the proof of Theorem~\ref{thm:main}, which is given in \S \ref{ssec:arithpf}.

\medskip

Let $G = \SU(n,1)$, and suppose that $\Gam < G$ is a lattice such that the complex hyperbolic orbifold $K \bs G / \Gam$ contains infinitely many maximal geodesic subspaces. Pass to a subsequence so that they are all either real or complex hyperbolic subspaces of the same type. By Lemma \ref{lem:TGtomsr}\eqref{num:TGtomsr3}, there is a fixed standard subgroup $W < G$ with normalizer $N$ and elements $g_i \in G$ so that
\[
\Del_i = N^{g_i^{-1}} \cap \Gam,
\]
is a lattice in $N^{g_i^{-1}}$ associated with the $i^{th}$ maximal geodesic subspace. Specifically, $\Del_i$ acts on the totally geodesic subspace $K \bs K N g_i^{-1}$ of $K \bs G$ with finite covolume. We will need the following general result.

\begin{prop}[Cf.\ Prop.\ 3.3 \cite{BFMS}]\label{prop:setup}
In addition to the above assumptions, suppose that $k$ is a local field and $\bfH$ is a connected adjoint $k$-algebraic group. If $\rho : \Gam \to \bfH(k)$ is a representation so that $\rho(\Del_i)$ has proper Zariski closure in $\bfH(k)$ for infinitely many $i$, then there is a $k$-vector space $V$ of dimension at least two, an irreducible representation of $\bfH$ on $V$, and a $W$-invariant, ergodic measure $\nu$ on the bundle $(G \times \bbP(V)) / \Gam$ that projects to Haar measure on $G / \Gam$.
\end{prop}

\begin{proof}
The proof is almost exactly the same as \cite[Prop.\ 3.3]{BFMS}, so we only sketch the proof. The assumption on $\rho(\Del_i)$ implies that we can pass to a subsequence to assume that the Zariski closures of the $\rho(\Del_i)$ are all contained in proper, nontrivial, $k$-algebraic subgroups $\bfJ_i < \bfH$ all of which have the same dimension $d > 0$. Let $\frakh$ be the Lie algebra of $H$ over $k$ and $\frakj_i \subset \frakh$ the subalgebra associated with $\bfJ_i$. We then take $V$ to be the $d^{th}$ exterior power of $\frakh$, so each $\frakj_i$ defines a point in $\bbP(V)$.

Let $S_i$ be the closed subgroup of $N^{g_i^{-1}}$ determined by Lemma \ref{lem:TGtomsr}\eqref{num:TGtomsr3}. Then $\Del_i$ is a lattice in $S_i$. Consequently, for each $i$ we obtain a measurable section $S_i / \Del_i \to (G \times \bbP(V)) / \Gam$. Let $\nu_i$ be the push-forward of Haar measure on $S_i / \Del_i$ and $\nu$ be an ergodic component of the weak-$*$ limit of the $\nu_i$. Ratner's theorem, the assumption that the geodesic submanifolds are distinct and maximal, and Proposition \ref{prop:BFMS3.1} implies that $\nu$ is a $W$-invariant, $W$-ergodic measure that projects to Haar measure on $G / \Gam$.
\end{proof}

\subsection{The proof of Theorem~\ref{thm:main}}\label{ssec:arithpf}

Suppose that $\Gam < \SU(n,1)$ is a lattice so that $\bbB^n / \Gam$ contains infinitely many maximal geodesic subspaces. Let $\ell$ be the adjoint trace field of $\Gam$ and $\bfH$ be the connected adjoint algebraic group associated with $\Gam$ as in \S\ref{sec:Hodge}. Given a place $v$ of $\ell$, let $\ell_v$ be the completion of $\ell$ at this place and $\rho_v : \Gam \to \bfH(\ell_v) = H_v$ be the natural inclusion. We have a place $v_0$ so that $H_{v_0} \cong \PU(n,1)$ and $\rho_{v_0}(\Gam)$ is the lattice embedding. To prove that $\Gam$ is arithmetic we must show that $\rho_v(\Gam)$ is precompact for all $v \neq v_0$.

Assume $\rho_v(\Gam)$ is not precompact for some $v\neq v_0$. Let $k=\ell_v$, and $W$ and $\{\Del_i\}$ be as in \S\ref{ssec:arithpfsetup}. We clearly have that $\rho_v(\Del_i)$ has proper Zariski closure in $H_v$. Then Proposition~\ref{prop:setup} applies to produce a vector space $V$ as in the conclusion of the proposition and a $W$-invariant measure $\nu$ on $(G \times \bbP(V)) / \Gam$ that projects to Haar measure on $G / \Gam$.

Then Theorem~\ref{thm:SRht} along with Propositions~\ref{prop:compk} and~\ref{prop:comppu} imply that either the pair $(k, H_v)$ is compatible with $G$, or $k = \bbR$ and $H_v \cong \PU(n,1)$. Since $\rho_v(\Gam)$ is unbounded, Theorem~\ref{thm:SR} applies and $\rho_v$ extends to a continuous homomorphism $\wh{\rho}_v : G \to H_v$, and this is a contradiction as explained in \cite[\S 3.2]{BFMS}. Therefore $\rho_v(\Gam)$ must be precompact for all $v\neq v_0$, which proves that $\Gam$ is arithmetic. \qed

\section{Examples, other results, and final comments}\label{sec:Final}

\subsection{Submanifolds of arithmetic manifolds}\label{ssec:arithex}

In this section, we briefly give three examples exhibiting some of the possibilities for geodesic submanifolds of arithmetic quotients of $\bbB^n$.

\medskip

The general construction is as follows. See \cite{TitsSS}. Let $\ell$ be a totally real number field and $\ell^\prime$ a totally imaginary quadratic extension of $\ell$. Suppose that $D$ is a central simple division algebra over $\ell^\prime$ of degree $d$ admitting an involution $\tau$ of \emph{second kind}, i.e., so that the restriction of $\tau$ to $\ell^\prime$ is the nontrivial Galois involution of $\ell^\prime / \ell$. For $r \ge 1$, $\tau$ extends to an anti-involution of the matrix algebra $\M_r(D)$ by $\tau$-conjugate transposition, denoted $x \mapsto x^*$. Note that under any embedding of $\ell^\prime$ in $\bbC$, this involution extends to complex conjugate transposition on $\M_r(D) \otimes_{\ell^\prime} \bbC \cong \M_{rd}(\bbC)$.

An element $h \in \M_r(D)$ is called \emph{$\tau$-hermitian} if $h^* = h$, and then we can define the $\ell$-algebraic unitary group $\bfG$ with $\ell$-points
\[
\bfG(\ell) = \{x \in \SL_r(D)\ :\ x^* h x = h \}.
\]
Choosing a maximal order $\calO$ of $D$, we obtain an arithmetic group
\[
\Gam_\calO^1 = \{x \in \SL_r(\calO)\ :\ x^* h x = h \} < \bfG(\ell).
\]
Let $n = rd - 1$. Choosing $D$ and $h$ so that
\[
\mathrm{Res}_{\ell / \bbQ}(\bfG)(\bbR) \cong \SU(n, 1) \times \SU(n+1)^{[\ell : \bbQ] - 1},
\]
one has that the projection of $\Gam_\calO^1$ to $\SU(n,1)$ is an arithmetic lattice. For any $n \ge 2$, all arithmetic subgroups of $\SU(n,1)$ are commensurable with some such $\Gam_\calO^1$.

\begin{ex}\label{ex:Simple}
When $d = 1$, we have $D = \ell$. Then $h$ is a $\tau$-hermitian form on an $\ell^\prime$-vector space $V$ of dimension $n + 1$, $\calO$ is the ring of integers of $\ell^\prime$, and $\Gam_\calO^1$ is sometimes called an arithmetic lattice of \emph{simplest type}. For concreteness, we take $\ell = \bbQ$, $\ell^\prime = \bbQ(\al)$ with $\al^2 = -1$, and $h$ to be the hermitian form fixed in \S\ref{subsec:simple}. Then
\[
\Gam_\calO^1 = \{g \in \SL_{n+1}(\bbZ[\al])\ :\ g^* h g = h \},
\]
is the subgroup of $\SL_{n+1}(\bbZ[\al])$ preserving the form $h$ and $\bfG$ is a $\bbQ$-algebraic group with $\bfG(\bbR) \cong \SU(n,1)$.

We claim that $\bbB^n / \Gam_\calO^1$ contains all possible types of geodesic submanifolds of a complex hyperbolic $n$-manifold. Indeed, let $\{e_i\}$ be the standard basis for our vector space $V \cong \bbC^{n+1}$. Restricting the form to the span of $\{e_1, \dots, e_m, e_{n+1}\}$ visibly gives an arithmetic subgroup $\Lam_\calO^1 < \SU(m,1)$ contained in $\Gam_\calO^1$, where $\SU(m,1)$ denotes the standard subgroup in the notation of \S\ref{subsec:simple}. This leads to properly immersed totally geodesic complex hyperbolic $m$-submanifolds for all $m$. Considering the real span of this subspace instead, the restriction of $h$ now defines a quadratic form on $\bbR^{n+1}$ stabilized by the standard $\SO_0(m,1)$ subgroup, hence one similarly finds real hyperbolic submanifolds of every real dimension between $2$ and $n$.
\end{ex}

\begin{ex}\label{ex:Kottwitz}
At the other extreme, assume that $d = n+1$ is prime, so $\bfG(\ell)$ is a subgroup of the group $\SL_1(D)$ of units of $D$ with reduced norm $1$. Fix a maximal order $\calO$ of $D$ and consider the arithmetic group $\Gam_\calO^1 < \SU(n,1)$. For example, in the case $n = 2$ all fake projective planes arise from this construction \cite{CartwrightSteger}.

Were $\bbB^n / \Gam_\calO^1$ to contain a proper geodesic subspace that is complex hyperbolic, then we would obtain an injection $\M_r(D^\prime) \hookrightarrow D$, where $D^\prime$ is a central simple $\ell_0^\prime$-division algebra of degree $d^\prime$ for some totally complex subfield $\ell^\prime_0$ of $\ell^\prime$ whose intersection $\ell_0$ with $\ell$ is totally real. Moreover, $r d^\prime$ divides $n+1$.

Since $n+1$ is prime, we claim that $r = 1$ and $d^\prime = n+1$. Indeed, the case $r = d$ and $d^\prime = 1$ is impossible since the algebra $\M_{n+1}(\ell_0^\prime)$ contains subalgebras of degree $1 < e < n+1$. After taking the tensor product with $\ell^\prime$, this contradicts the fact that $D$ has prime degree and hence contains no such subalgebras. Now we rule out the case that $\ell_0^\prime$ is a proper subfield of $\ell^\prime$. To see this, note that in this case the $\ell$-algebraic unitary group $\bfG$ associated with $D$ now contains $\bfH \otimes_{\ell_0} \ell$, where $\bfH$ is the $\ell_0$-algebraic unitary group associated with $D^\prime$. However, $\bfH$ is noncompact at exactly one real place of $\ell_0$ and $\bfG$ is consequently noncompact at exactly $[\ell : \ell_0]$ real places of $\ell$. Since $\bfG$ is noncompact at exactly one real place of $\ell$, we have that $\ell = \ell_0$ and $D = D^\prime$.

It follows that $\bbB^n / \Gam_\calO^1$ contains no complex hyperbolic geodesic subspaces. A similar argument shows that $\bbB^n / \Gam_\calO^1$ also contains no real hyperbolic subspaces of any dimension $m \ge 2$. In other words, the only properly immersed geodesic subspaces of $\bbB^n / \Gam_\calO^1$ are the closed geodesics.
\end{ex}

\begin{rem}
We leave it to the reader to show that the case $\M_r(D)$ for $r > 1$ and $D$ of possibly composite degree $d > 1$ gives intermediate examples living between the previous two.
\end{rem}

\subsection{Geodesic submanifolds and the intersection pairing}\label{ssec:AG}

In this section, we give some background on the algebraic and complex geometry of complex hyperbolic manifolds and explain the proof of Theorem~\ref{thm:mainAG}. For simplicity we restrict to the case complex dimension two, then give references for the analogous results in higher dimension at the end of the section.

If $M$ is a closed complex hyperbolic $2$-manifold, then the associated complex hyperbolic manifold $M$ is a smooth projective surface of general type whose Chern numbers satisfy $c_1^2(M) = 3\, c_2(M)$ \cite{HirzebruchProp}. We recall that $c_1^2(M)$ is the self-intersection of the canonical divisor $K_M \in H_2(M)$ and $c_2(M)$ is the Euler number of $M$. Moreover, Yau's famous solution to the Calabi conjecture says that this equality of Chern numbers holds for surfaces of general type if and only if $M = \bbB^2 / \Gam$ for some torsion-free cocompact lattice $\Gam$ in $\PU(2,1)$, where $\bbB^2$ denotes the unit ball in $\bbC^2$ with its complex hyperbolic metric \cite{Yau}. It is a well-known consequence of Selberg's lemma that any cocompact lattice $\Gam < \PU(2,1)$ contains a finite index subgroup $\Gam^\prime$ so that $\bbB^2 / \Gam^\prime$ is a manifold, hence there is no loss of generality in assuming this is the case.

In the noncompact setting, as described in detail in the proof of Theorem \ref{thm:SRht}, we can replace a lattice $\Gam < \PU(2,1)$ with a subgroup of finite index so that $M = \bbB^2 / \Gam$ is a manifold of the form $X \ssm D$, where $X$ is a smooth projective surface and $D$ is a certain smooth divisor on $X$. In this setting, one has the \emph{logarithmic} canonical divisor $K_X + D \in H_2(X)$, where $K_X$ is the canonical divisor on $X$. To unify the two cases, when $M$ is compact we consider $D$ to be the empty divisor and so $X = M$.

\begin{proof}[Proof of Theorem~\ref{thm:mainAG}]
Applying a very general theorem due to Borel \cite[Thm.~3.7]{BorelMetric} to the compactification \cite{AMRT, Mok}, an immersed totally geodesic complex hyperbolic submanifold of $M$ determines an irreducible complex curve $C_0$ on $M$ that compactifies to a curve $C$ on $X$. Hirzebruch--H\"ofer relative proportionality \cite[\S B.3]{HirzebruchHofer} is precisely the statement that a projective curve $C$ on $X$ determines a totally geodesic subspace of $M$ if and only if Equation \eqref{eq:Proportionality} holds with respect to the divisor $K_X + D$. See \cite[Thm.\ 0.1]{MSVZ} for our statement, which is slightly more general than the original. In particular, the theorem is an immediate consequence of Theorem \ref{thm:main} and Corollary \ref{cor:main}.
\end{proof}

\begin{rem}
Margulis asked whether arithmeticity is detected purely by the topology of the locally symmetric space. Theorem \ref{thm:mainAG} implies that if $M$ contains a totally geodesic curve, then arithmeticity of $M$ is completely determined by the restriction of the intersection pairing on $H_2(M)$ to the curves on $M$. In other words, arithmeticity is detected by the topology of the underlying variety.
\end{rem}

We now give references for where one can give a precise version of our statements in higher dimensions. For the equality of (logarithmic) Chern numbers that characterizes higher-dimensional complex hyperbolic manifolds, see \cite{Tsuji}. A general version of relative proportionality that uniquely determines complex hyperbolic totally geodesic submanifolds was given by M\"{u}ller-Stach, Viehweg, and Zuo. See Theorem 2.3 and Addendum 2.4 in \cite{MSVZ}.

\medskip

This interpretation of our main results leads to the following question.

\begin{qtn}
Can one classify the totally geodesic curves on nonarithmetic Deligne--Mostow orbifolds?
\end{qtn}

This seems particularly approachable in dimension two in the sense that the underlying spaces for these orbifolds are closely related to blowups of the complex projective plane. In particular, one can connect geodesic curves to classical plane curves, where immersed geodesic curves will have singularities arising from self-intersections. One can then use relative proportionality to detect which curves are totally geodesic.

\subsection{The proof of Theorem~\ref{thm:Siu}}\label{ssec:Siu}

\begin{proof}[Proof of Theorem~\ref{thm:Siu}]
With notation as in the statement of the theorem, suppose that $f : M \to N$ is a surjective mapping so that $f(Z_i)$ is contained in a proper geodesic subspace of $N$ for each $i$. By hypothesis, the induced map on fundamental groups induces a Zariski dense homomorphism $\rho : \Gam \to \PU(m,1)$, where $\Gam$ is the lattice in $\SU(n,1)$ associated with $M$ and $d = \dim(N)$. Let $\Lam < \PU(d,1)$ be the fundamental group of $N$.

We now proceed exactly as in Proposition \ref{prop:setup}. Let $\Del_i < \Gam$ be the subgroup associated with $Z_i$. The assumption on $f$ implies that $\rho(\Del_i)$ is not a Zariski dense subgroup of $\PU(d,1)$. In particular, we pass to an infinite subsequence so that each each $Z_i$ is a geodesic submanifold of $M$ of the same type with associated with the standard subgroup $W$ of $G$ and $\rho(\Del_i)$ has Zariski closure contained in some conjugate of a fixed proper, nontrivial, positive-dimensional $J < \PU(d,1)$.

Let $S_i$ be the closed subgroup of $G$ associated with $\Del_i$ by Lemma \ref{lem:TGtomsr}\eqref{num:TGtomsr3}. The appropriate exterior power of the Lie algebra of $H$ defines a vector space $V$ such that for each $i$ we can construct a measurable section from $S_i / \Del_i$ to $(G \times \bbP(V)) / \Gam$ in order to build a $W$-invariant, $W$-ergodic measure on the bundle that projects to Haar measure on $G / \Gam$. This and Proposition~\ref{prop:comppu} allow us to apply Theorem~\ref{thm:SR}\eqref{thm:SRcomp} to conclude that $\rho$ extends to a continuous homomorphism from $\SU(n, 1)$ to $\PU(d, 1)$. Therefore, $d = n$ and $\rho(\Gam)$ is a lattice in $\PU(n,1)$. Since $\Lam$ is also a lattice and $\rho(\Gam) \le \Lam$, we see that $\rho(\Gam)$ is a finite index subgroup. It follows that $f$ is homotopic to a cover. This proves the theorem.
\end{proof}

\subsection{The proof of Theorem~\ref{thm:notTG}}\label{ssec:NotTG}

\begin{proof}[Proof of Theorem~\ref{thm:notTG}]
To prove the theorem we must take care to differentiate between $\Tr\Ad$ for $\SL_2(\bbC)$ and $\PO_0(3,1)$, since lattices in $\PU(3,1)$ naturally contain lattices in $\PO_0(3,1)$, but results on nonintegral traces for Kleinian groups typically are stated for $\SL_2(\bbC)$. We denote the two adjoint representations by $\Ad_\bbC$ and $\Ad_\bbR$, respectively, as they denote the traces of the adjoint representation for $\fraks\frakl_2(\bbC)$ considered as a complex (resp.\ real) vector space.

There are many known hyperbolic $3$-manifolds $M$ with nonintegral traces. See \cite[\S 5.2.2]{MaclachlanReid} for a closed example and \cite{Cheseblo} for a hyperbolic link complement. However, in the literature this means that for a given lift of $\Gam = \pi_1(M)$ from $\PSL_2(\bbC)$ to $\SL_2(\bbC)$ there is a $\gam \in \Gam$ with a nonintegral trace for the associated $2 \times 2$ matrix, i.e., that $\Tr(\gam)$ is not an algebraic integer. The embedding of $\Gam$ in $\PSL_2(\bbC)$ is the lattice embedding associated with the complete hyperbolic structure, and we can choose any lift to $\SL_2(\bbC)$, as $\gam$ will have nonintegral trace for any chosen lift. In what follows, we identify $\gam$ with a matrix in $\SL_2(\bbC)$.

Let $\lam^{\pm 1}$ be the eigenvalues of $\gam$. Then $\lam^{\pm 1}$ both have minimal polynomial $t^2 - \Tr(\gam) t + 1$, and it follows that $\lam$ and $\lam^{-1}$ also are not algebraic integers. A direct calculation gives:
\begin{align*}
\Tr(\Ad_\bbC(\gam)) &= \lam^2 + \lam^{-2} + 1 \\
&= \Tr(\gam)^2 - 1
\end{align*}
It follows easily from the assumption that $\Tr(\gam)$ is not an algebraic integer that $\Tr(\Ad_\bbC(\gam))$ is also not an algebraic integer.

Now, suppose that $M$ is a geodesic submanifold of a complex hyperbolic $n$-manifold. When $n = 3$, the inclusion $\Gam < \pi_1(M)$ gives
\[
\Gam < \PO_0(3,1) < \PU(3,1).
\]
When $n > 3$, we instead have
\[
\Gam < \SO_0(3,1) < \PO_0(n,1) < \PU(n,1),
\]
where we have chosen a lift of $\Gam$ from $\PO_0(3,1)$ to $\SO_0(3,1)$. In each case, we want to show that $\Tr(\Ad_{\fraks\fraku(n,1)}(\gam))$ is not an algebraic integer. This will contradict Theorem~\ref{thm:SRht}(3) and complete the proof of the theorem.

For $n = 3$, we can conjugate and choose our hermitian form so that $\gam$ maps to the image in $\PO_0(3,1) < \PU(3,1)$ of the diagonal matrix with entries $\{\lam^2, 1, 1, \lam^{-2}\}$. A direct calculation from the root space decomposition for $\fraks\fraku(3,1)$ gives:
\begin{align*}
\Tr(\Ad_{\fraks\fraku(3,1)}(\gam)) &= 4 \lam^2 + 4 \lam^{-2} + \lam^4 + \lam^{-4} + 5 \\
&= (\lam^2 + \lam^{-2})^2 + 4 (\lam^2 + \lam^{-2}) + 3
\end{align*}
Since the above expression is a monic integral polynomial in $\lam^2 + \lam^{-2}$, which is not an algebraic integer, we conclude that $\Tr(\Ad_{\fraks\fraku(3,1)}(\gam))$ is also not an algebraic integer. We leave it to the reader to verify that the same proof works for higher $n$ where $3$ is now a sum of roots of unity and $4$ is changed to $2(n-1)$.
\end{proof}

We note that Theorem~\ref{thm:notTG} does not apply for large families of hyperbolic $3$-manifolds. For example, if $M$ is a finite-volume hyperbolic $3$-manifold commensurable with a non-Haken hyperbolic $3$-manifold, then $\pi_1(M)$ has all traces integral \cite[\S5.2]{MaclachlanReid}. More concretely, if $K \subset S^3$ is a \emph{small knot}, then $\Tr(\Ad(\gam))$ is an algebraic integer for all $\gam \in \pi_1(S^3 \ssm K)$. The figure-eight knot is the only knot whose complement is arithmetic \cite{ReidKnot}, and one can show that it is commensurable with a totally geodesic submanifold of a complex hyperbolic manifold arising from the construction in Example~\ref{ex:Simple}. It follows that if $K$ is any other hyperbolic knot and $M$ is a complex hyperbolic manifold containing $S^3 \ssm K$ as a geodesic submanifold, then $M$ is necessarily nonarithmetic. We know of no such example.

\begin{qtn}
Which knot (or link) complements are isometric to a totally geodesic submanifold of a complex hyperbolic $n$-manifold for some $n \ge 3$? Are there combinatorial/diagrammatic obstructions to a knot complement being a geodesic submanifold of a complex hyperbolic manifold?
\end{qtn}

There are only two known commensurability classes of nonarithmetic lattices in $\PU(3,1)$ \cite{Deraux3d}. The known geometric constructions of orbifolds in each commensurability class might allow one to at least find some examples. Of course, Theorem~\ref{thm:main} implies that only finitely many commensurability classes of knot and link complements could be associated with each commensurability class of nonarithmetic lattices in $\PU(3,1)$.

\bibliographystyle{abbrv}
\bibliography{Biblio}

\end{document}